\documentclass[12pt]{amsart}
\usepackage{amsmath,times,epsfig,amssymb,amsbsy,amscd,amsfonts,amstext,color,bm}
\usepackage[margin=1in]{geometry}
\usepackage{enumerate}
\usepackage{placeins}
\usepackage{array}
\usepackage{tabulary}
\usepackage{multirow}
\usepackage{graphicx}
\usepackage{hyperref}
\usepackage{hhline}
\usepackage{soul}
\usepackage{multirow}
\usepackage[table, svgnames, dvipsnames]{xcolor}
\usepackage{makecell, cellspace, caption}
\usepackage{subcaption}
\usepackage[color,matrix,arrow]{xypic}
\usepackage{tikz}
        \usetikzlibrary{arrows.meta}
        \usetikzlibrary{arrows}
        \usetikzlibrary{quotes}
        \usetikzlibrary{graphs}
        \usetikzlibrary{positioning}
        \usetikzlibrary{decorations.markings}
\usepackage{tikz-cd}
\usepackage[color=red]{todonotes}
\hypersetup{
  colorlinks   = true, 
  urlcolor     = blue, 
  linkcolor    = blue, 
  citecolor   = red 
}

\theoremstyle{plain}
\newtheorem{theorem}{Theorem}[section]
\newtheorem{corollary}[theorem]{Corollary}
\newtheorem{lemma}[theorem]{Lemma}
\newtheorem{claim}[theorem]{Claim}

\makeatletter
    
    \@addtoreset{equation}{section}
  \makeatother

\theoremstyle{definition}
\newtheorem{definition}[theorem]{Definition}
\newtheorem{notation}[theorem]{Definition \& Notation}
\newtheorem{example}[theorem]{Example}
\newtheorem{conjecture}[theorem]{Conjecture}
\newtheorem{observation}[theorem]{Observation}
\theoremstyle{remark}
\newtheorem{remark}[theorem]{Remark}

\newcommand{\A}{\mathcal{A}}
\newcommand{\B}{\mathcal{B}}
\newcommand{\C}{\mathbb{C}}

\newcommand{\CC}{\mathsf{CC}}
\newcommand{\NC}{\mathsf{NC}}
\newcommand{\CO}{\mathsf{CO}}
\newcommand{\LC}{{2\text{-}\mathsf{LCO}}}
\newcommand{\LS}{2\text{-}\mathsf{LS}}
\newcommand{\I}{\mathsf{ID}}

\newcommand{\R}{\mathbb{R}}
\newcommand{\K}{\mathbb{K}}

\newcommand{\scS}{\mathcal{S}}

\newcommand{\Z}{\mathbb{Z}}

\newcommand{\h}{{\rm h}}
\newcommand{\vn}{\noindent}

\newcommand{\dsc}{{\rm dsc}}

\newcommand{\M}{\mathcal{M}}
\newcommand{\tbf}{\textbf} 
\newcommand{\cc}{\mathbf{c}}

\newcommand{\quasi}{\operatorname{quasi}}

\newcommand{\codim}{\operatorname{codim}}
\newcommand{\Mat}{\operatorname{Mat}}
 
\DeclareMathOperator{\Der}{Der}

\newcolumntype{K}[1]{>{\centering\arraybackslash}p{#1}}

\begin{document}

\title[Worpitzky-compatible and Shi-free arrangements]{Worpitzky-compatible sets and the freeness of arrangements between Shi and Catalan}

\date{\today}

\begin{abstract}
Given an irreducible root system, the Worpitzky-compatible subsets are defined by a geometric property of the alcoves inside the fundamental parallelepiped of the root system. 
This concept is motivated and  mainly understood through a lattice point counting formula concerning the characteristic and Ehrhart quasi-polynomials. 
In this paper, we show that the Worpitzky-compatibility has a simple combinatorial characterization in terms of roots. 
As a byproduct, we obtain a complete characterization by means of Worpitzky-compatibility for the freeness of the arrangements interpolating between the extended Shi and Catalan arrangements. 
This is a completion of the earlier result by Yoshinaga in 2010 which was done for simply-laced root systems. 
 \end{abstract}

\author{Takuro Abe}
\address{Takuro Abe, Department of Mathematics, Rikkyo University, 3-34-1 Nishi-Ikebukuro, Toshima-ku, 1718501 Tokyo, Japan.}
\email{abetaku@rikkyo.ac.jp}
 
\author{Tan Nhat Tran}
\address{Tan Nhat Tran, Institut f\"ur Algebra, Zahlentheorie und Diskrete Mathematik, Fakult\"at f\"ur Mathematik und Physik, Leibniz Universit\"at Hannover, Welfengarten 1, D-30167 Hannover, Germany.}
\email{tan.tran@math.uni-hannover.de}


\date{\today}
\maketitle

\section{Introduction}
\label{sec:intro}

\subsection{Back ground and motivation}
\label{subsec:mov}

Let   $V=\R^\ell$ with the standard inner product $(\cdot,\cdot)$.
 Let $\Phi$ be an irreducible (crystallographic) root system in $V$. 
Let $\Delta := \{\alpha_1,\ldots,\alpha_\ell \}$ be a set of simple roots of $\Phi$ and $\Phi^+$ the positive system associated to $\Delta$. 
 For $n \in \Z$ and $\alpha \in  \Phi^+$, define an affine hyperplane $H_{\alpha}^n := \{x \in V \mid (\alpha,x) = n\}$ in $V$. 
For a hyperplane arrangement $\A$ in $V$, denote by $\cc \A$ the \emph{cone} of $\A$ (see \S\ref{subsec:free}).
 
  \begin{definition}
 \label{def:Shi+-}
  For a nonnegative integer $k \in \Z_{\ge 0}$ and a subset $\Sigma\subseteq\Phi^+$, define the following hyperplane arrangement in $V$:
$$
 \scS^k_{\Sigma} = \scS^k_{\Sigma} (\Phi) := \{ H_{\alpha}^n \mid \alpha \in \Phi^+, 1-k \le n \le k\} \cup \{ H_{\alpha}^{-k} \mid  \alpha \in \Sigma\}.
$$
The subset $\Sigma$ is called \tbf{Shi-free} (resp.~\tbf{free}) if the cone $\cc\scS^k_{\Sigma}$ is a free arrangement for every $k>0$  (resp.~for $k=0$).
\end{definition}

Free arrangements are defined formally in Definition \ref{def:free-arr}. 
In brief, an arrangement is called  \tbf{free} if its  \emph{module of logarithmic derivations} is a free module.  
The (Shi-)freeness of root systems has been a central topic in the study of free arrangements for decades. 
For simply-laced (type $ADE$) root systems, a characterization for the Shi-freeness is known due to Yoshinaga \cite{Yo10} (Theorem \ref{thm:SL-CC}).
The ultimate goal of this paper is to complete this characterization for all root systems  (Theorem \ref{thm:SF-characterize}).
First let us give more information about the freeness of $\cc \scS^k_{\Sigma}$. 

 \begin{enumerate}[(1)] 
\item Let $k=0$. When $\Sigma = \Phi^+$, the arrangement $\A_{\Phi^+}:= \scS^0_{\Phi^+}$ is known as the \tbf{Weyl arrangement} of $\Phi$. 
For arbitrary $\Sigma$, $\A_\Sigma:= \scS^0_{\Sigma}$ is a subarrangement of $\A_{\Phi^+}$. 
The Weyl arrangement is a well-known free arrangement, e.g.~\cite{S93}, \cite[Theorem 6.60]{OT92}. 
If the root system $\Phi$ is of type $A$, then $\A_\Sigma$ can be identified with a \emph{graphic arrangement} (see \S\ref{sec:remark}) whose freeness is completely characterized by \emph{chordal graphs} \cite{St72, ER94}. 
Apart from type $A$, the freeness of  arbitrary $\Sigma$ is unknown in general. 

\vskip .5em

\item Let $k>0$. When $\Sigma = \emptyset$ and $\Sigma = \Phi^+$, the arrangements $\mathrm{Shi}^{[1-k,k]}_{\Phi}:= \scS^k_{\emptyset}$ and $\mathrm{Cat}^{[-k,k]}_{\Phi}:= \scS^k_{\Phi^+}$ are known as the  \tbf{extended Shi arrangement} and \tbf{extended Catalan arrangement}, respectively. 
Thus the arrangement $ \scS^k_{\Sigma}$, when $\Sigma$ varies, can be regarded as an interpolation between the extended Shi and Catalan arrangements. 
The freeness of $\cc\mathrm{Shi}^{[1-k,k]}_{\Phi}$ and $\cc\mathrm{Cat}^{[-k,k]}_{\Phi}$ had been conjectured by Edelman-Reiner \cite{ER96} until they were affirmatively settled by Yoshinaga \cite{Yo04}. 

\vskip .5em

\item The most significant class for which the Shi-freeness is known to be true for any root system is that of the \emph{ideals}. 
 The \tbf{root poset}  $(\Phi^+, \ge)$ is the poset with partial order defined by $\beta_1 \ge \beta_2$ if $\beta_1-\beta_2 \in\sum_{i=1}^\ell \Z_{\ge 0}\alpha_i$. 
A  subset $\Sigma\subseteq\Phi^+$ is called an \tbf{ideal} if for $\beta_1,\beta_2 \in \Phi^+$, $\beta_1 \ge \beta_2, \beta_ 1 \in \Sigma$ implies $\beta_2 \in \Sigma$. 
Then for any ideal $\Sigma$ and $k \ge0$, the cone $\cc \scS^k_{\Sigma}$ is always free. 
The case $k=0$ was first partially proved by Sommers-Tymoczko \cite{ST06} and  later completely settled by Abe-Barakat-Cuntz-Hoge-Terao \cite{ABCHT16}. 
The case $k>0$ was done in a follow-up paper of Abe-Terao \cite{AT16}. 

\vskip .5em

\item There is another arrangement closely related to $  \scS^k_{\Sigma}$. Define
$$
  \scS^k_{-\Sigma} := \{ H_{\alpha}^n \mid \alpha \in \Phi^+, 1-k \le n \le k\} \setminus \{ H_{\alpha}^{k} \mid  \alpha \in \Sigma\}.
  $$
Abe-Terao \cite{AT16} showed that  if $k>0$, then $\cc\scS^k_{\Sigma}$ and  $ \cc\scS^k_{-\Sigma}$ share the freeness, i.e.~$\cc\scS^k_{\Sigma}$ is free if and only if  $ \cc\scS^k_{-\Sigma}$ is free (Theorem \ref{thm:exp}). 
If this occurs for some $k>0$, then $\A_\Sigma=\scS^0_{\Sigma}$ is also free. 
Thus the freeness of $\Sigma$ is a necessary (but not sufficient) condition for its Shi-freeness. 

\vskip .5em

\item Towards a search for a full characterization of the Shi-freeness, it is essential to extend the class of ideals. 
A  subset $\Sigma\subseteq\Phi^+$ is called \tbf{coclosed} if for any $\alpha \in \Sigma $ and $\beta_1,\beta_2 \in\Phi^+ $ such that $\alpha =d_1\beta_1+d_2\beta_2$ with $d_1,d_2 \in \Z_{> 0}$, either $\beta_1 \in \Sigma$ or $\beta_2  \in \Sigma$. 
It is easy to see that every ideal of a root system is coclosed. 
For simply-laced root systems, Yoshinaga  showed that the coclosedness is the missing piece of a sufficient condition for the Shi-freeness.
 \end{enumerate}
 
 \begin{theorem}{\cite[Theorem 5.1]{Yo10}}
 \label{thm:SL-CC}
Let $\Phi$ be an irreducible root system of type $ADE$ and $\Sigma\subseteq\Phi^+$. 
Then $\Sigma$ is Shi-free if and only if  $\Sigma$ is free and coclosed. 
\end{theorem}

However, the theorem above is not always true for doubly-laced root systems. 
In this paper, we complete the characterization for every root system by replacing the coclosed sets by a more general concept, the so-called \emph{Worpitzky-compatible} sets due to Ashraf-Tran-Yoshinaga \cite{ATY20}. 
The appearance of the Worpitzky-compatibility here is interesting and unexpected as this concept has original motivation from a geometric property of alcoves of root system and a lattice point counting problem seemingly unrelated to the freeness. 

\subsection{The main results}
\label{subsec:results}
To state the results formally, let us first recall the concept of compatibility. 
A connected component of $V \setminus \bigcup_{ \alpha\in \Phi^+,n \in \Z} H_{\alpha}^n$ is called an \tbf{alcove}.
The closure of an alcove is an $\ell$-simplex. 
A  \tbf{face} of a simplex is the convex hull of any subset of its vertices. 
A  \tbf{facet} of an $\ell$-simplex is a face that is an $(\ell-1)$-simplex.
By abuse of notation, when we say a face of an alcove we mean a face of its closure.
Let $A$ be an alcove. 
A \tbf{wall} of $A$ is a hyperplane that supports a facet of $A$. 
The \tbf{ceilings} of $A$ are the walls which do not pass through the origin and have the origin on the same side as $A$. 
The  \tbf{upper closure} $A^\diamondsuit$  of $A$ is the union of $A$ and its facets supported by the ceilings of $A$. 
We will often abuse notation and call $A^\diamondsuit$ an upper closed alcove (though it is not an alcove). 
Let  $P^\diamondsuit:=\{x\in V \mid 0<(\alpha_i ,x)\le 1\,(1 \le i \le \ell)\}$ be the \tbf{fundamental parallelepiped} (of the coweight lattice) of $\Phi$. 
Then $P^\diamondsuit$ has the following partition: 
$$P^\diamondsuit=\bigsqcup_{A:\, \text{alcove},\,A \subseteq P^\diamondsuit}A^\diamondsuit,
$$ 
which is known as the \tbf{Worpitzky partition}, e.g. \cite[Proposition 2.5]{Y18W}, \cite[Exercise 4.3]{H90}.

 \begin{definition}{{\cite[Definition 4.8]{ATY20}}}
 \label{def:compatibleW}   
A subset $\Sigma \subseteq \Phi^+$ is called \tbf{Worpitzky-compatible} in $\Phi$, or \tbf{compatible} for short, if for each alcove $A\subseteq P^\diamondsuit$, the intersection $A^\diamondsuit \cap H_{\alpha}^{n_\alpha}$ of its upper closure $A^\diamondsuit$ and any affine hyperplane $ H_{\alpha}^{n_\alpha}$ for $\alpha\in \Sigma,n_\alpha \in \Z$ is either empty, or contained in a ceiling $H_{\beta}^{n_\beta}$ of $A$ for some $\beta\in \Sigma,n_\beta \in \Z$. 
In short, every nonempty intersection can be lifted to a facet intersection.
\end{definition}
 
 In particular, the empty set $\emptyset$ and the positive system $\Phi^+$ itself are always compatible. 
The compatibility was originally defined in order to make a counting formula concerning the \emph{characteristic} and \emph{Ehrhart quasi-polynomials} valid (Theorem \ref{thm:CO}). 
It is proved that every coclosed subset is compatible \cite[Proof of Theorem 4.16]{ATY20} (see also  \cite[Theorem 6]{TT21}). 
Furthermore, when the root system is of type $A$, the converse of the previous fact is also true and these properties can be characterized by  \emph{cocomparability graphs}  \cite[Theorems 2 and 9]{TT21}. 
 
The first main result in this paper is a characterization of the compatibility by a root theoretic argument and a local property of the compatibility itself. 
On the one hand, the root theoretic argument demonstrates a more direct combinatorial relationship of the compatibility and coclosedness. 
On the other hand, the local property gives a key reason why the compatibility appears in the Shi-freeness characterization. 

We need a few more notations and definitions. 
For an arrangement $\A$ in $V$, denote by $L(\A)$ the \emph{intersection poset} of $\A$.
Set $L_p(\A):=\{X \in L(\A) \mid \codim(X)=p\}$ for  $0 \le p \le \ell$.
\begin{notation}\label{not:associated}
 Let $\Phi$ be an irreducible root system and let $\A:=\A_{\Phi^+}$ be the Weyl arrangement of $\Phi$. 
A subset of $\Phi$ is a \tbf{root subsystem} if it is a root system in its own right. 
If $X \in L_p(\A)$, then $\Phi_X := \Phi \cap X^{\perp}$ is a rank $p$ root subsystem (not necessarily irreducible) of $\Phi$. 
A positive system of $\Phi_X$ is taken to be $\Phi^+_X:=\Phi^+\cap\Phi_X$.
Let $\Delta_X$ be the set of simple roots of  $\Phi_X$ associated to  $\Phi^+_X$. 
For a subset $\Sigma\subseteq\Phi^+$, denote  $\Sigma_X:=\Sigma \cap \Phi_X^+$. 
We call $\Phi_X$ and $\Sigma_X$ the \textbf{localizations} of $\Phi$ and $\Sigma$ on $X$, respectively.
\end{notation}

 \begin{definition}
 \label{def:c-co}
A subset $\Sigma\subseteq\Phi^+$ is called
 \begin{enumerate}[(a)] 
\item  \tbf{negatively coclosed} if for any $\alpha \in \Sigma $ and $\beta_1,\beta_2 \in\Phi^+ $ such that $\alpha =d_1\beta_1+d_2\beta_2$ with $d_1,d_2 \in \Z_{> 0}$ and $(\beta_1,\beta_2)<0$, either $\beta_1 \in \Sigma$ or $\beta_2  \in \Sigma$,
 \item \tbf{$2$-locally compatible} if for any $X \in L_2(\A)$ such that $\Phi_X$ is irreducible, the localization $\Sigma_X=\Sigma \cap \Phi^+_X$ is compatible in $\Phi_X$,
  \item \tbf{$2$-locally simple} if for any $X \in L_2(\A)$ such that $\Phi_X$ is irreducible,  either  $\Sigma_X $ contains a simple root of $\Phi_X$ (i.e.~ $\Sigma_X \cap \Delta_X \ne \emptyset$) or $\Sigma_X=\emptyset$. 
 \end{enumerate}
\end{definition}

In the subsequent characterizations, we must distinguish some particular subsets of positive roots in a root system of type $G_2$.

 \begin{definition}
 \label{def:G2}
Given a root system $\Phi=G_2$ with $\Delta=\{\alpha_1,\alpha_2\}$ where $\alpha_2$ is the unique long simple root, define the following subsets $\Sigma\subseteq\Phi^+$:
   \begin{enumerate}[(a)] 
   \item \label{item:CO-WC} 
   $\Sigma=\{\alpha_2\} \cup S$ with $\emptyset \ne S \subseteq \{ 2\alpha_1+\alpha_2, 3\alpha_1+\alpha_2, 3\alpha_1+2\alpha_2\}$. 
   \item  \label{item:LS-CO}  
   $\Sigma=\{\alpha_1, 3\alpha_1+2\alpha_2\}\cup S$ with $S \subseteq \{  \alpha_1+\alpha_2, 2\alpha_1+\alpha_2,\}$. 
\end{enumerate}
\end{definition}

We are ready to state our first main result connecting the compatibility, a geometric property of alcoves and the negative coclosedness, a combinatorial property of roots.

\begin{theorem}
 \label{thm:WC}
 Let $\Phi$ be an irreducible root system and $\Sigma\subseteq\Phi^+$. The following are equivalent. 
   \begin{enumerate}[(1)] 
   \item $\Sigma$ is compatible. 
   \item $\Sigma$ is $2$-locally compatible.
      \item One of the following occurs:
   \begin{enumerate}[(i)] 
   \item If $\Phi \ne G_2$, $\Sigma$ is  negatively coclosed. 
   \item If $\Phi = G_2$, $\Sigma$ is  negatively coclosed, or one of the seven exceptions in Definition \ref{def:G2}(\ref{item:CO-WC}).
 \end{enumerate}
  \end{enumerate}
\end{theorem}

Our second main result is a generalization of Theorem \ref{thm:SL-CC} to any root system.

\begin{theorem}
 \label{thm:SF-characterize}
 Let $\Phi$ be an irreducible root system and $\Sigma\subseteq\Phi^+$. The following are equivalent. 
   \begin{enumerate}[(1)] 
   \item $\Sigma$ is Shi-free. 
   \item $\Sigma$ is free and $2$-locally simple.
   \item One of the following occurs:
   \begin{enumerate}[(i)] 
   \item If $\Phi \ne G_2$, $\Sigma$ is compatible and  free. 
   \item If $\Phi = G_2$, $\Sigma$ is compatible, or one of the four exceptions in Definition \ref{def:G2}(\ref{item:LS-CO}).
 \end{enumerate}
 \end{enumerate}
\end{theorem}

We emphasize that the proofs of Theorems \ref{thm:WC} and \ref{thm:SF-characterize} require only the classification of all rank $2$ root systems ($A_1^2, A_2, B_2=C_2, G_2$), and the fact that given a root system $\Phi \ne G_2$, any rank $2$ irreducible root subsystem of $\Phi$ is of type $A_2$ or $B_2$.

  \begin{remark}
\label{rem:LC-LS}
Given a root system $\Phi$, denote by $\I$, $\CC$, $\NC$, $\CO$, $\LC$, $\LS$ the set of all ideals, coclosed, negatively coclosed, compatible, $2$-locally compatible, $2$-locally simple sets in $\Phi$, respectively. 
By the theorems above, we have the following relations between these concepts:
$$\I \subseteq \CC \subseteq \NC  \subseteq \CO= \LC \subseteq \LS.$$

For any containment relation above, there exists an example that makes it strict. 
Let us add a few more combinatorial and geometric insights.
   \begin{enumerate}[(a)] 
   \item \label{item:NC-CC} The containment $\CC \subseteq \NC$ (i.e.~every coclosed subset is negatively coclosed) is clear from definition. 
   If $\Phi$ is simply-laced, then $\CC = \NC$ since any rank $2$ irreducible root subsystem of $\Phi$ is of type $A_2$. 
   Let $\Phi=B_2$ and suppose $\Delta =\{\alpha_1,\alpha_2\}$ with the long simple root $\alpha_2$. 
Then $\Sigma = \{2\alpha_1+\alpha_2, \alpha_2\}$ is the unique subset of $\Phi^+$ such that $\Sigma\in \NC \setminus \CC$. 
The reason is that although $2\alpha_1+\alpha_2 =  \alpha_1+ (\alpha_1+\alpha_2)$, the negative coclosedness does not require  $ \alpha_1$ or $\alpha_1+\alpha_2$ to be in $\Sigma$ since these roots are orthogonal.

   \item  \label{item:A2B2}  
Let $\Phi=A_2$ or $B_2$ with $\Delta=\{\alpha_1,\alpha_2\}$. Then
$$\NC= \CO= \LC = \LS.$$
The second equality is obvious. 
The compatibility in these cases can be easily verified by two dimensional pictures. 
For a positive root $\alpha = \sum_{i=1}^\ell d_i \alpha_i\in \Phi^+$, the \tbf{height} of $\alpha$ is defined by $ {\rm ht}(\alpha) :=\sum_{i=1}^\ell  d_i$. 
Let $\{\varpi^\vee_1, \ldots ,\varpi^\vee_\ell\}$ be the dual basis of $\Delta$, namely, $(\alpha_i ,\varpi^\vee_j)=1$ if $i=j$ and $0$ otherwise.
Then $H_{\alpha}^{ {\rm ht}(\alpha)}\cap P^\diamondsuit =  \left\{ \sum_{i=1}^\ell \varpi^\vee_i\right\}$ for any $\alpha\in \Phi^+$. 
In type $A_2$ or $B_2$, 
an affine hyperplane orthogonal to a positive root intersects an upper closed alcove inside $ P^\diamondsuit$ at a non-facet intersection only if the intersection is the point $v = \varpi^\vee_1 +\varpi^\vee_2$. 
The point $v$ is  a vertex of the alcove ``furthest away" from the origin, i.e.~ the alcove with ceilings $H_{\alpha_1}^1$, $H_{\alpha_2}^1$. 
By the preceding calculation, $v$ is contained in every hyperplane of the form $H_{\alpha}^{ {\rm ht}(\alpha)}$ for $\alpha\in \Phi^+$. 
See Figure \ref{fig:A2} for an illustration in type $A_2$ (and \cite[Figure 2]{Y18W} for type $B_2$). 

Thus a subset $\Sigma\subseteq\Phi^+$ is compatible if and only  either $\Sigma$ is empty or $\Sigma$ contains a simple root, i.e.~$\Sigma $ is $2$-locally simple. 
This property also characterizes the negative coclosedness. 
Hence these concepts must be the same when $\Phi=A_2$ or $B_2$. 

\item In general, given a subset $\Sigma\subseteq\Phi^+$, it is very difficult to check whether $\Sigma$ is compatible or not by using the primary definition of compatibility or the characterization by quasi-polynomials in Theorem \ref{thm:CO}. 
The characterization of the compatibility by negative coclosedness in our Theorem \ref{thm:WC}, however, gives a very simple and effective way to do so. 
See \ref{ex:F4} for an example. 
 \end{enumerate}
\end{remark}

\section{Preliminaries}
\label{sec:free-arr}

\subsection{Free arrangements}
\label{subsec:free}
We begin by recalling some basic concepts and preliminary results of free arrangements. Our standard reference is \cite{OT92}.
Let $ \mathbb{K} $ be a field and let $V=  \mathbb{K}^{\ell} $. 
A \textbf{hyperplane} in $V$ is an affine subspace of codimension $1$ of $V$.
An \textbf{arrangement} is a finite collection of hyperplanes in  $ V $. 
An arrangement is called \textbf{central} if every hyperplane in it passes through the origin. 

Let  $ \mathcal{A} $ be an arrangement.
Define the \textbf{intersection poset} $ L(\mathcal{A}) $ of $ \mathcal{A} $ by 
\begin{align*}
L(\mathcal{A})  :=\left\{\bigcap_{H \in \mathcal{B}}H \neq \emptyset\,\middle\vert\,  \mathcal{B} \subseteq \mathcal{A} \right\},
\end{align*}
where the partial order is given by reverse inclusion: $X\le Y$ if  $Y\subseteq X$ for $X, Y \in L(\A)$. 
We agree that $ V $ is a unique minimal element in $ L(\mathcal{A}) $ as the intersection over the empty set.

For each $X \in L(\A)$, define the \textbf{localization} of $\A$ on $X$  by 
$${\A}_X  :=\{ K \in {\A} \mid X \subseteq K\} \subseteq \A,$$
and the \textbf{restriction} ${\A}^{X}$ of ${\A}$ to $X$ by 
$${\A}^{X}  :=\{ K \cap X  \ne \emptyset \mid K \in{\A }\setminus {\A}_X\}.$$

Let $ \{x_{1}, \dots, x_{\ell}\} $ be a basis for the dual space $V^{\ast} $ and let $ S  :=\mathbb{K}[x_{1}, \dots, x_{\ell}] $. 
The \textbf{defining polynomial} $Q(\A)$ of $\A$ is given by
$$Q(\A) :=\prod_{H \in \A} \alpha_H \in S,$$
where $ \alpha_H=a_1x_1+\cdots+a_\ell x_\ell+d$  $(a_i, d \in \mathbb{K})$ satisfies $H = \ker \alpha_H$.

The \textbf{cone} $ \textbf{c}\A$ of $\A$ is the central arrangement in $\mathbb{K}^{\ell+1}$ with the defining polynomial
$$Q(\textbf{c}\A) :=z\prod_{H \in \A} {}^h\alpha_H \in \mathbb{K}[x_1,\ldots, x_\ell,z],$$
where $ {}^h\alpha_H :=a_1x_1+\cdots+a_\ell x_\ell+dz$ is the homogenization of $\alpha_H$, and $z=0$ is the  hyperplane at infinity. 

A $\K$-linear map $\theta:S\longrightarrow S$ which satisfies $\theta(fg) = \theta(f)g + f\theta(g)$ is called a \tbf{derivation}.
Let $\Der(S)$ be the set of all derivations of $S$. 
It is a free $S$-module with a basis $\{\partial/\partial x_1,\ldots,\partial/\partial x_{\ell}\}$ consisting of the usual partial derivatives.
 
\begin{definition}{{\cite[Definitions 4.5 and 4.15]{OT92}}}
\label{def:free-arr}
Let $\A$ be a central arrangement in  $V=\K^\ell$. 
The \textbf{module $D(\mathcal{A}) $ of logarithmic derivations}  is defined by 
	\begin{equation*}
	D(\A)  := \{ \theta \in \Der(S) \mid \theta(\alpha_H) \in \alpha_HS \mbox{ for all } H \in \A\}.
	\end{equation*}
	We say that $\A$ is \tbf{free} if the module $D(\A)$ is a free $S$-module. 
\end{definition}

The freeness can be extended to a more general class of  \emph{multiarrangements}.
A \tbf{multiarrangement} is a pair $(\A, m)$ where $\A$ is a central arrangement in $\K^\ell$ and $m$ is a map $m : \A \longrightarrow \Z_{\ge0}$, called \tbf{multiplicity}. 
Let $(\A, m)$ be a multiarrangement.
The defining polynomial $Q(\A,m)$ of $(\A, m)$ is given by 
$$Q(\A,m) := \prod_{H \in \A} \alpha^{m(H)}_H \in S.$$
When $m(H) = 1$ for every $H \in \A$, $(\A,m)$ is simply a hyperplane arrangement. 
The \tbf{module $D(\A,m)$ of logarithmic derivations} of  $(\A, m)$ is defined by
$$D(\A,m) :=  \{ \theta\in \Der(S) \mid \theta(\alpha_H) \in \alpha^{m(H)}_HS \mbox{ for all } H \in \A\}.$$
 
We say that $(\A, m)$  is \tbf{free}  if $D(\A,m)$ is a free $S$-module. 
It is known that  $(\A, m)$ is always free for $\ell\le 2$ \cite[Corollary 7]{Z89}.

Let $H \in \A$. The \tbf{Ziegler restriction} $(\A^H,m^H)$ of $\A$ onto $H$ is a multiarrangement defined by 
$$m^H(X):= |\A_X|-1 \quad \mbox{for } X \in \A^H.$$ 
We say that $\A$ is \tbf{$3$-locally free along $H$} if the localization $\A_X$ is free for all $X\in L_3(\A)$ with $X\subseteq H$.

\begin{theorem}{{\cite[Theorem 2.2]{Yo04}}, {\cite[Theorem 4.1]{AY13}}, {\cite[Theorem 11]{Z89}}}
\label{thm:Yoshinaga's criterion}
Let $\A$ be a central arrangement in $\K^\ell$ with $\ell \ge 3$ and let $H\in \A$. Then
$\A$ is free if and only if the Ziegler restriction  $(\A^H,m^H)$ is free and $\A$ is  $3$-locally free along $H$.
\end{theorem}
\subsection{Characteristic quasi-polynomials and Worpitzky-compatibility}
\label{subsec:quasi}
Next we recall the definition of the \emph{characteristic quasi-polynomial} of an integral arrangement following \cite{KTT08, KTT11}. 
The main motivation in \cite{ATY20} for defining the Worpitzky-compatibility is to study this quasi-polynomial for arrangements arising from root systems. 

A function $\varphi: \Z \longrightarrow \C$ is called a \tbf{quasi-polynomial} if there exist a positive integer $\rho\in\Z_{>0}$ and polynomials $f^k(t)\in\mathbb{Q}[t]$ ($1 \le k \le \rho$) such that for any $q\in\Z_{>0}$ with  $q\equiv k\bmod \rho$, 
\begin{equation*}
\varphi(q) =f^k(q).
\end{equation*}
The number $\rho$ is called a \tbf{period}, and the polynomial $f^k(t)$ is called the \tbf{$k$-constituent} of the quasi-polynomial $\varphi$.

Let $\ell,n\in\Z_{>0}$ be positive integers. 
Denote by $\Mat_{\ell\times n}(\Z)$ the set of all $\ell\times n$ matrices with integer entries. 
Let $C=(c_1,\ldots,c_n) \in \Mat_{\ell\times n}(\Z)$ with no zero columns and let $b=(b_1,\ldots,b_n) \in\Z^n$. 
Set $A  := \begin{pmatrix} C\\ b\end{pmatrix}  \in \Mat_{(\ell + 1)\times n}(\Z)$. 
The matrix $A$ defines the following hyperplane arrangement in $\R^\ell$, called \tbf{integral} arrangement 
$$\A=\A(A) := \{ H_j \mid 1\le j\le n\},
\,\mbox{ where }
H_j  := \{x\in \R^\ell\mid xc_j=b_j\}.$$

Let $q\in\Z_{>0}$ and $\Z_q  :=\Z/q\Z$. 
 For $a \in \Z$, let $\overline{a}  := a + q\Z \in \Z_q$ denote the $q$-reduction of $a$. 
For a matrix or vector $A'$ with integral entries, denote by $\overline{A'}$ the entry-wise $q$-reduction of $A'$.
The \tbf{$q$-reduction} $\A_q$ of $\A$ is defined by 
$$\A_q := \{ H_{j,q} \mid 1\le j\le n\},
\,\mbox{ where }
H_{j,q}  := \{ z \in \Z_q^\ell \mid z\overline{c_j} =\overline{b_j}\}.$$

Denote $\Z_q^{\times} :=\Z_q \setminus \{\overline0\}$.
The \tbf{complement} $\M(\A_q)$ of $\A_q$ is defined by 
$$\M(\A_q)  := \Z_q^\ell \setminus\bigcup_{j=1}^n H_{j,q} = \{z\in \Z_q^\ell \mid z\overline{C} -\overline{b} \in (\Z_q^{\times})^n\}.$$

\begin{theorem}{{\cite[Theorem 2.4]{KTT08}, \cite[Theorem 3.1]{KTT11}}}
\label{thm:KTT}
There exists a monic quasi-polynomial $\chi^{\quasi}_{\A}(q)$ of degree $\ell$ such that  for sufficiently large $q$,
$$|\M(\A_q) | = \chi^{\quasi}_{\A}(q).$$ 
This quasi-polynomial is called the \tbf{characteristic quasi-polynomial} of $\A$. 
\end{theorem}

The name ``characteristic quasi-polynomial" is made by inspiration of the fact that the $1$-constituent of $\chi^{\quasi}_{\A}(q)$ coincides with the \emph{characteristic polynomial} (of the intersection poset) of $\A$ \cite[Remark 3.3]{KTT11}.

Now  let $\Phi$ be an irreducible  root system in $\R^\ell$, with a fixed set of simple roots $\Delta= \{\alpha_1,\ldots,\alpha_\ell \}$ and the associated positive system $\Phi^+ \subseteq \Phi$. 
For a subset $\Sigma\subseteq\Phi^+$, let $C_\Sigma$ be the \emph{coefficient matrix} of $\Sigma$ with respect to $\Delta$, i.e.~$C_\Sigma = (C_{ij})$ is the $\ell \times |\Sigma|$ integral matrix that satisfies 
$$\Sigma = \left\{\sum_{i=1}^\ell C_{ij}\alpha_i \,\middle\vert\ 1 \le j \le |\Sigma|  \right\}.$$
The central integral arrangement $\A(C_\Sigma)$ defined by  $C_\Sigma$ is \emph{linearly equivalent} to the Weyl subarrangement $\A_\Sigma$ defined by $\Sigma$. 
Here the linear equivalence means there exists  an invertible endomorphism of $\R^\ell$ that maps the hyperplanes of one onto the hyperplanes of the other. 

It is natural to ask for which $\Sigma\subseteq\Phi^+$ the characteristic quasi-polynomial $\chi^{\quasi}_{\Sigma}(q): =\chi^{\quasi}_{\A(C_\Sigma)}(q)$ can be computed by means of some invariants of $\Phi$. 
The study in \cite{ATY20} shows that a Worpitzky-compatible subset $\Sigma$ is such an example. 

Let $\tilde{\alpha} \in \Phi^+$ denote the highest root, i.e.~the unique maximal element in the root poset $(\Phi^+, \ge)$. 
Then $\tilde{\alpha}$ can be written uniquely as $\tilde{\alpha}= \sum_{i=1}^\ell c_i\alpha_i$  with all $c_i \in  \Z_{>0}$. 
Denote $\alpha_0 := -\tilde{\alpha}$ and $c_0:=1$. 
The sum $\h:= \sum_{i=0}^\ell c_i$ is called the \tbf{Coxeter number} of $\Phi$.
Recall that $\{\varpi^\vee_1, \ldots ,\varpi^\vee_\ell\}$ denotes the dual basis of $\Delta$. 
The \tbf{coweight lattice} of $\Phi$ is defined by $Z(\Phi^\vee):=\bigoplus_{i=1}^\ell \Z\varpi^\vee_i  \simeq \Z^\ell$.  
Then the  fundamental parallelepiped of  the coweight lattice is given by  
$$P^\diamondsuit = \sum_{i=1}^\ell (0,1]_\R \varpi^\vee_i.$$ 
Let $A^\circ$ be the \tbf{fundamental alcove} of $\Phi$, then its closure $\overline{A^\circ}=\mathrm{conv} \left\{0, \frac{\varpi^\vee_1}{c_1},\ldots, \frac{\varpi^\vee_\ell}{c_\ell}\right\} \subseteq  \overline{P^\diamondsuit}$ can be regarded as a rational polytope in $\bigoplus_{i=1}^\ell \R\varpi^\vee_i \simeq \R^\ell$. 
The counting function 
$${\rm L}_{\overline{A^\circ}}(q) := |q\overline{A^\circ} \cap Z(\Phi^\vee) |$$ 
is a quasi-polynomial in $q$, which is known as the \tbf{Ehrhart quasi-polynomial} of $\overline{A^\circ} $ w.r.t.~the  coweight lattice. 

Let $W$ be the \emph{Weyl group} of $\Phi$. 
For $\Sigma \subseteq \Phi^+$,  set $\Sigma^c:=\Phi^+ \setminus \Sigma$. 
The \tbf{descent $\dsc_\Sigma$ w.r.t.~$\Sigma$}  is a function $\dsc_\Sigma: W \longrightarrow \Z_{\ge0}$ defined by
$$
\dsc_\Sigma(w) := \sum_{0 \le i \le \ell,\, w(\alpha_i)\in -\Sigma^c}c_i.
$$
Let $f$  be the \tbf{index of connection} of $\Phi$. 
The \tbf{$\A$-Eulerian polynomial} $E_\Sigma(t)$ of $\Sigma$ is defined by
$$E_\Sigma(t):= \frac1f \sum_{w\in W}t^{\h-\dsc_\Sigma(w)}.$$
It is proved in \cite[Theorem 4.7]{ATY20} that $E_\Sigma(t)$ is a polynomial with all positive integer coefficients.

\begin{theorem}{\cite[Theorems 4.11 and 4.24]{ATY20}}
 \label{thm:CO}
 Let $\Phi$ be an irreducible root system and $\Sigma\subseteq\Phi^+$. 
 Suppose $E_\Sigma(t) = \sum_{i=0}^na_i t^i$.
 The following are equivalent. 
   \begin{enumerate}[(1)] 
   \item $\Sigma$ is compatible. 
   \item For every $q\in\Z_{>0}$, 
   $$\chi^{\quasi}_{\Sigma}(q) = \sum_{i=0}^n a_i {\rm L}_{\overline{A^\circ}}(q-i).$$
   \item The generating function of $\chi^{\quasi}_{\Sigma}(q)$ is given by
   $$ \sum_{q \ge 1}\chi^{\quasi}_{\Sigma}(q)t^q = \frac{E_\Sigma(t)}{\prod_{i=0}^\ell (1-t^{c_i}) }.$$
 \end{enumerate}
\end{theorem}

\begin{example}
\label{ex:A2}
Let $\Phi=A_2$ with $\Delta=\{\alpha_1,\alpha_2\}$ depicted in  Figure \ref{fig:A2}. 
The Worpitzky partition of $P^\diamondsuit$ induces a partition of the $q$-dilation $qP^\diamondsuit\cap Z(\Phi^\vee)$ intersected with the coweight lattice. 
 Let  $\Sigma_0=\emptyset$, $\Sigma_1=\{\alpha_1, \alpha_1+\alpha_2\}$ and $\Sigma_2=\{\alpha_1+\alpha_2\}$. 
 The empty set is always compatible. 
 By Remark \ref{rem:LC-LS}(\ref{item:A2B2}), $\Sigma_1$ is compatible while  $\Sigma_2$ is not. 
By definition,  for all $q >0$
\begin{align*}
 \chi^{\quasi}_{\Sigma_0}(q) &= \left| \Z_q^2 \right|  =q^2,\\
 \chi^{\quasi}_{\Sigma_1}(q) &=\left|\{ z\in \Z_q^2\mid z_1,z_1+z_2 \ne \overline0\} \right| = (q-1)^2, \\
  \chi^{\quasi}_{\Sigma_2}(q) &=\left|\{ z\in \Z_q^2\mid z_1+z_2\ne \overline0\} \right| = q(q-1). 
\end{align*}
The fundamental alcove is given by $\overline{A^\circ}=\mathrm{conv} \{0,  \varpi^\vee_1, \varpi^\vee_2\} $. 
Hence its Ehrhart quasi-polynomial is given by ${\rm L}_{\overline{A^\circ}}(q)=\frac{(q+1)(q+2)}2.$ 
Moreover, one may compute the $\A$-Eulerian polynomials: $E_{\Sigma_0}(t) =t^2+t$, $E_{\Sigma_1}(t) =t^3+t^2$ and $E_{\Sigma_2}(t) =t^3+t$ (by e.g.~ a graphical method in \cite[\S3]{TT21}).
Thus 
\begin{align*}
 \chi^{\quasi}_{\Sigma_0}(q) &={\rm L}_{\overline{A^\circ}}(q-2) + {\rm L}_{\overline{A^\circ}}(q-1), \\
 \chi^{\quasi}_{\Sigma_1}(q) &={\rm L}_{\overline{A^\circ}}(q-3) + {\rm L}_{\overline{A^\circ}}(q-2), \\
  \chi^{\quasi}_{\Sigma_2}(q) &={\rm L}_{\overline{A^\circ}}(q-3) + {\rm L}_{\overline{A^\circ}}(q-1)-1. 
\end{align*}
The calculation above is consistent with Theorem \ref{thm:CO}.
\end{example}

\begin{figure}[!ht]
\centering
    \includegraphics[width=9.5cm,height=9cm]{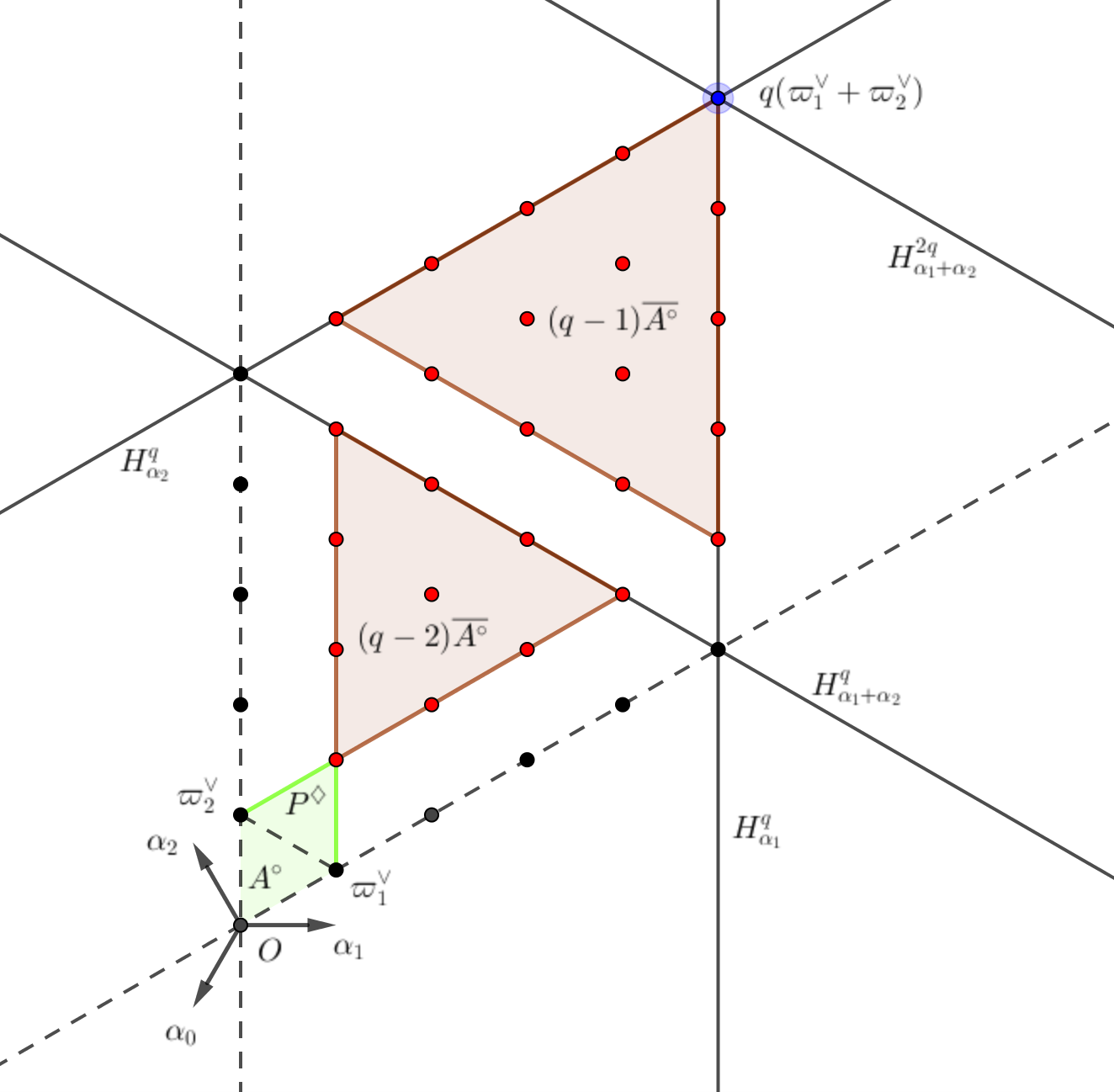}
    \caption{Root system of type $A_2$ from Example \ref{ex:A2}.}
    \label{fig:A2}
\end{figure}

\begin{remark}
\label{rem:quasi-Shi}
It follows from Theorem \ref{thm:CO}  that the compatibility is an essential geometric property for the characteristic quasi-polynomial of a Weyl subarrangement to be expressible in terms of the Ehrhart quasi-polynomial. 
Such an expression is also important for the study on the characteristic quasi-polynomials of deformed Weyl arrangements. 

For instance,  by using the Ehrhart theoretic method, Yoshinaga \cite[Theorem 5.1]{Y18W} showed that the characteristic quasi-polynomial of the extended Shi arrangement is actually a polynomial. 
More precisely,  
$$\chi^{\quasi}_{\mathrm{Shi}_{\Phi}^{[1-k,k]}}(q)  
 =  (q-k\h)^\ell \quad \mbox{for every $q\in\Z_{>0}$}.$$
The formula above provides the first example for the \emph{period collapse} phenomenon studied in \cite{HTY23}. 
 
The Ehrhart theoretic approach is also used to affirmatively settle the ``Riemann hypothesis", a conjecture of Postnikov-Stanley \cite{SP00} that all roots of the characteristic polynomial of the \emph{extended Linial arrangement} have the same real part \cite{Y18W, Y18L, Tam23}. 
\end{remark}

\section{Proof of the first main result: Theorem \ref{thm:WC}}
\label{sec:proof1}

 First we recall several known properties of roots. 
\begin{lemma} {\cite[Lemma 3.1]{PR01}}
\label{lem:proximity} 
Let $\beta \in \Phi^+$ and $\alpha, \alpha' \in \Delta$ with $\alpha \ne \alpha' $. 
If $\beta -\alpha \in \Phi^+$ and $\beta -\alpha' \in \Phi^+$,
then either $\beta =\alpha +\alpha'$ or $\beta -\alpha -\alpha' \in \Phi^+$. 
\end{lemma}

Using a similar argument as in the proof of Lemma \ref{lem:proximity}, one may show the following extension of it.
(The case $\Phi = G_2$ should be treated separately.)

\begin{lemma} 
\label{lem:gen-proximity} 
Suppose $\beta,\gamma_1, \gamma_2 \in\Phi^+$ with $ \gamma_1\ne \gamma_2  $ and $ \gamma_1 - \gamma_2 \notin\Phi$. 
If $\beta - \gamma_1 \in \Phi^+$ and $\beta - \gamma_2 \in \Phi^+$,
then either $\beta = \gamma_1 + \gamma_2$ or $\beta - \gamma_1 - \gamma_2 \in \Phi$. 
\end{lemma}

\begin{lemma}{\cite[Lemma 11.10]{LN04}}
\label{lem:3roots} 
Suppose $\beta_1,  \beta_2, \beta_3 \in \Phi$ with $\beta_1+  \beta_2+ \beta_3 \in \Phi$ and $\beta_i+  \beta_j \ne 0$ for $i \ne j$. Then at least two of the three partial sums $\beta_i+  \beta_j$ belong to $ \Phi$.
\end{lemma}

 \begin{lemma} {\cite[Lemma 3.2]{S05}}
\label{lem:reorder} 
Let $\beta_1 \in \Phi \cup\{0\}$. 
Suppose that $\beta_i \in \Phi^+$ for $2 \le i \le k $ and $\sum^k_{i=1} \beta_i \in \Phi \cup\{0\}$. Then there exists a re-ordering of the $\beta_i$'s with $i \ge 2$ so that
$\sum^j_{i=1} \beta_i \in \Phi \cup\{0\}$   for all $1 \le j \le k$.
\end{lemma}

For any alcove $A$ and  $\gamma\in \Phi^+$, there exists a unique integer $r$ with $r-1 < (x,\gamma) < r$ for all $x \in A$. We denote this integer by $r(A,\gamma)$. 
The map $r(A,-) : \Phi^+ \longrightarrow \Z$ is called the \tbf{address} of the alcove $A$.

 \begin{lemma}{\cite[Theorem 5.2]{Shi87}, \cite[Lemma 2.4]{A05}}
\label{lem:Shi}
Suppose that for each $\gamma\in \Phi^+$ we are given an integer $r(\gamma)$. 
The map $r:\Phi^+ \longrightarrow \Z$ is the address of some alcove $A$ if and only if 
$$r(\gamma)+r(\gamma')-1 \le r(\gamma+\gamma') \le r(\gamma)+r(\gamma')\mbox{ whenever } \gamma,\gamma',\gamma+\gamma'\in \Phi^+.$$
\end{lemma} 

For a subset $B \subseteq \Phi$, define $\Phi_B(\Z) := \Phi \cap \Z B$. 
Then $\Phi_B(\Z)$ is a root subsystem of $\Phi$. 
A positive system of $\Phi_B(\Z)$ is taken to be $ \Phi^+ \cap \Phi_{B}(\Z)$.

 \begin{lemma}{\cite[Proof of Theorem 4.16]{ATY20}}
\label{lem:ideal-crucial}
Let $A  \subseteq P^\diamondsuit$ be an alcove inside the fundamental parallelepiped.  
Suppose there exist $\alpha\in \Phi^+,n_\alpha \in \Z$ such that the intersection $A^\diamondsuit \cap H_{\alpha}^{n_\alpha}$ is a  nonempty face   of $A^\diamondsuit$. 
Let $H_{\beta_1}^{n_{\beta_1}}, \ldots, H_{\beta_m}^{n_{\beta_m}}$ for $m\ge1$ be the (pairwise distinct) ceilings of $A$ that define the intersection, i.e.~ $A^\diamondsuit \cap H_{\alpha}^{n_\alpha}= \bigcap_{i=1}^m H_{\beta_i}^{n_{\beta_i}}\cap  A^\diamondsuit$. 
Then $B :=\{\beta_i \mid 1 \le i \le m\}$  is a set of simple roots of $\Phi_{B}(\Z)$, and $\alpha \in \Phi^+ \cap \Phi_{B}(\Z)$. 
\end{lemma}

 In order to make the proof of Theorem \ref{thm:WC} more readable, we break it into three lemmas.

\begin{proof}[\textbf{Proof of Theorem \ref{thm:WC}}]
We show $(2) \Leftrightarrow (3)$, $(2) \Rightarrow (1)$ and $(1) \Rightarrow (2)$ in Lemmas \ref{lem:(2)<=>(3)}, \ref{lem:(2)=>(1)}  and \ref{lem:(1)=>(2)}, respectively. 
The implication $(1) \Rightarrow (2)$  is the most difficult part. 
 \end{proof}

\begin{lemma} 
\label{lem:(2)<=>(3)} 
A subset $\Sigma\subseteq\Phi^+$ is $2$-locally compatible if and only if it is  negatively coclosed, or one of the seven exceptions in type $G_2$ in Definition \ref{def:G2}(\ref{item:CO-WC}).
\end{lemma}

\begin{proof} 
We call a subset $\Sigma\subseteq\Phi^+$  \tbf{$2$-locally negatively coclosed} if for any $X \in L_2(\A)$, the localization $\Sigma_X = \Sigma \cap \Phi_X^+$  is  negatively coclosed in $\Phi_X$. 
First observe that $\Sigma$ is  negatively coclosed if and only if it is $2$-locally negatively coclosed.
This is easy to see since the angle between two roots does change after taking localization; hence $(\beta_1,\beta_2)<0$ in $\Phi$  if and only if $(\beta_1,\beta_2)<0$ in $\Phi_X$.  

Suppose $\Phi \ne G_2$. 
Then any rank $2$ irreducible root subsystem $\Phi'$ of $\Phi$ is of type $A_2$ or $B_2$. 
 By Remark \ref{rem:LC-LS}(\ref{item:A2B2}), $\NC=\CO$ in $\Phi'$. 
 Therefore, for any  $\Sigma\subseteq\Phi^+$, 
 $$\Sigma \in \NC \Leftrightarrow \Sigma \mbox{ is $2$-locally negatively coclosed } \Leftrightarrow \Sigma\in \LC.$$
 
For $\Phi = G_2$, the assertion of  Lemma \ref{lem:(2)<=>(3)}  is proved by a direct check in Figure \ref{fig:G2}. 
The main reason why the exceptional cases exist is that for $\Sigma \in \NC$, if a root of the form $\alpha=d_1\alpha_1+d_2(\alpha_1+\alpha_2)$ with $d_1,d_2 \ge 1$ is in $\Sigma$, then either $\alpha_1 $ or $\alpha_1+\alpha_2  $ is in $\Sigma$ since $(\alpha_1,\alpha_1+\alpha_2)<0$. 
In particular, $\{\alpha_2, \alpha\} \notin \NC$. 
However, $\{\alpha_2, \alpha\} $ is compatible since the hyperplane $H_{\alpha_2}^1$ prevents 
the other hyperplanes from having a non-facet intersection with an upper closed alcove. 
\end{proof}

\begin{figure}[!ht]
\centering
    \includegraphics[width=11cm,height=9cm]{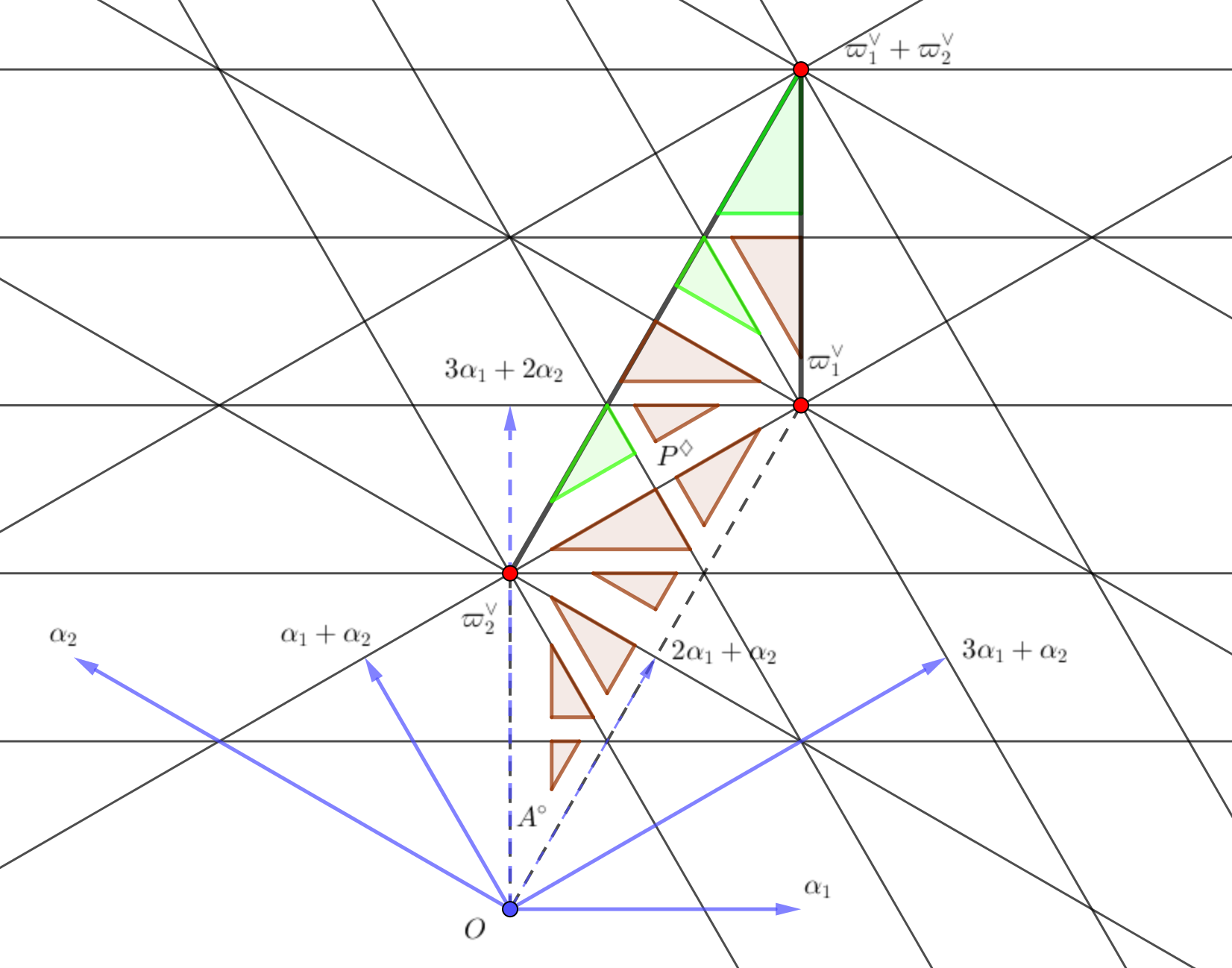}
    \caption{The Worpitzky partition of the fundamental parallelepiped $P^\diamondsuit$ in type $G_2$. 
    A non-facet intersection between an affine hyperplane and an upper closed alcove occurs only at the  alcoves in green. 
    }
    \label{fig:G2}
\end{figure}

\begin{lemma} 
\label{lem:(2)=>(1)} 
If a subset $\Sigma\subseteq\Phi^+$ is $2$-locally compatible, then it is  compatible.
\end{lemma}

\begin{proof} 
 Let $\Sigma\in \LC$. 
 Suppose that there exist an alcove $A\subseteq P^\diamondsuit$ and a hyperplane 
 $H_{\alpha}^{n_\alpha}$ for $\alpha\in \Sigma,n_\alpha \in \Z$ such that the intersection $A^\diamondsuit \cap H_{\alpha}^{n_\alpha}$ is nonempty and a non-facet of $A^\diamondsuit $. 
 Then we may write $A^\diamondsuit \cap H_{\alpha}^{n_\alpha}= \bigcap_{i=1}^m H_{\beta_i}^{n_{\beta_i}}\cap  A^\diamondsuit$ for $m\ge2, \beta_i\in \Phi^+,n_{\beta_i} \in \Z$ and $H_{\beta_i}^{n_{\beta_i}}$'s are the ceilings of $A$. 
 The assertion is proved once we show that $\beta_i \in \Sigma$ for some $i$. 
Note that by Lemma \ref{lem:ideal-crucial}, $B :=\{\beta_i \mid 1 \le i \le m\}$  is a set of simple roots of $\Phi_{B}(\Z)$, and $\alpha \in \Phi^+ \cap \Phi_{B}(\Z)$. 
In particular, $\alpha = \sum_{i=1}^m d_i\beta_i$ with at least two $d_{i_1}\ge 1$, $d_{i_2}\ge1$. 
Hence by Lemma \ref{lem:reorder}, there exists $1 \le k \le m$ such that $\gamma:=\alpha-\beta_k \in \Phi^+$. 

For non-triviality we may assume that the rank of $\Phi$ is at least $3$. 
Then any rank $2$ irreducible root subsystem   of $\Phi$ is of type $A_2$ or $B_2$. 
Let $X:=H_\gamma \cap H_{\beta_k} \in L_2(\A)$. 
Then $  \Phi_X$ is a  rank $2$ irreducible root subsystem   of $\Phi$ of type $A_2$ or $B_2$ containing $\{\alpha,\beta_k,\gamma\}$. 
Since $\Sigma\in \LC$, the localization  $\Sigma_X$ is compatible in $\Phi_X$.
 By Remark \ref{rem:LC-LS}(\ref{item:A2B2}), $\Sigma_X$ must contain a simple root of $\Phi_X$, i.e.~$\Sigma_X \cap \Delta_X \ne \emptyset$. 
 Since $\alpha =\beta_k+ \gamma$, either  $\beta_k \in \Delta_X$ or $\gamma \in \Delta_X$. 
 If $\beta_k \in \Sigma_X $, then $\beta_k \in \Sigma$ and we are done. 
 If $\gamma \in \Sigma_X $, then we may conclude the proof by repeating the argument above for $  \Sigma \ni \gamma = (d_k-1)\beta_k+  \sum_{i\ne k} d_i\beta_i$ in place of $\alpha$. 
 If $\beta_k, \gamma \notin  \Sigma_X $, then it must happen that $\Phi_X=B_2$, $\alpha$ is the highest root in $\Phi_X^+$, and either $\alpha - 2\beta_k \in \Sigma_X$ or $\alpha - 2\gamma \in  \Sigma_X$. 
 The latter cannot happen because otherwise, $\alpha - 2\gamma  = (2-d_k)\beta_k- \sum_{i\ne k} d_i\beta_i \in \Phi^+ \cap \Phi_{B}(\Z)$. 
 This implies that $\alpha - 2\gamma$ is a positive root in $ \Phi_{B}(\Z)$ but this is a contradiction since there exists $d_i$ with $i\ne k$ such that $d_i \ge 1$. 
 Thus $\gamma':= \alpha - 2\beta_k  \in \Sigma_X$, and we may conclude the proof by repeating the argument above for $ \gamma' = (d_k-2)\beta_k+  \sum_{i\ne k} d_i\beta_i$ in place of $\alpha$. 
\end{proof}

The following technical property of roots is the key ingredient in the proof of the implication $(1) \Rightarrow (2)$ in Theorem \ref{thm:WC}.
 \begin{lemma} 
\label{lem:address} 
Let $\Phi$ be an irreducible root system. 
Let $X\in L_2(\A_{\Phi^+})$ and suppose that  the localization $\Phi_X$ is irreducible. 
Denote  $\Delta_X=\{\gamma_1,\gamma_2 \}$ for distinct $\gamma_1,\gamma_2 \in \Phi^+$. 
Define a map $r=r_X:\Phi^+ \longrightarrow \Z$ inductively on height of positive roots as follows:
   \begin{enumerate}[(i)] 
   \item \label{lem:address1} $ r(\beta)=1$ if $\beta \in \Delta$ or  $\beta \le \gamma_1$ or  $\beta \le \gamma_2$.
   \item  \label{lem:address2} 
For $\beta  \notin \Delta$, $\beta \not\le \gamma_1$,  and $\beta \not\le \gamma_2$, $r(\beta) = r(\beta-\gamma_i)+1$  if $\beta-\gamma_i\in \Phi^+$ for $i=1$ or $2$.
\item \label{lem:address3} 
Otherwise, $r(\beta) = \max\{ r(\beta-\alpha) \mid \alpha \in \Delta \mbox{ is a simple root such that }\beta-\alpha \in \Phi^+ \}$.
\end{enumerate}
Then  $r:\Phi^+ \longrightarrow \Z$ is the address of some alcove $A$. 
\end{lemma}

\begin{proof} 
First we show that the map $r$ is indeed well-defined, i.e.~ $r(\beta)$ is uniquely determined for every $\beta \in \Phi^+$. 
 We argue by an induction on the height ${\rm ht}(\beta)$ of $\beta$. 
 The case $\beta \in \Delta$ is clear. 
 For $\beta \notin \Delta$, it suffices to show if $\beta - \gamma_1 \in \Phi^+$ and $\beta - \gamma_2 \in \Phi^+$, then $r(\beta-\gamma_1)=r(\beta-\gamma_2)$.
By Lemma \ref{lem:gen-proximity}, $\delta:=\beta - \gamma_1 - \gamma_2 \in \Phi \cup\{0\}$. 
If $\delta \in \Phi^- \cup\{0\}$, then $r(\beta-\gamma_1)=r(\beta-\gamma_2)=1$  by condition \ref{lem:address}\eqref{lem:address1}. 
If $\delta \in \Phi^+$, then $r(\beta-\gamma_1)=r(\beta-\gamma_2)=r(\delta)+1$ by condition \ref{lem:address}\eqref{lem:address2}. 
Since $r(\delta)$,  $r(\beta-\gamma_1)$, $r(\beta-\gamma_2)$ are uniquely determined by the induction hypothesis, the conclusion follows.

To show $r$ is the  address of an alcove, we use Lemma \ref {lem:Shi}. 
It suffices to show for any $\beta \in \Phi^+$ and any choice of $\gamma,\gamma'\in \Phi^+$ such that $\beta = \gamma+\gamma'$, the following inequalities hold:
\begin{equation} 
\label{eq:address}
 r(\gamma)+r(\gamma')-1 \le r(\beta) \le r(\gamma)+r(\gamma').
\end{equation} 
 
 We argue by an induction on ${\rm ht}(\beta)$. 
 The assertion is clear if $\beta \in \Delta$ or  $\beta \le \gamma_1$ or  $\beta \le \gamma_2$.
Assume $\beta \notin \Delta$, $\beta \not\le \gamma_1$,  $\beta \not\le \gamma_2$ and let $ \gamma,\gamma'\in \Phi^+$ be such that $\beta = \gamma+\gamma'$. 
Our  induction hypothesis is that  \eqref{eq:address} holds true for every $\delta \in \Phi^+$ with ${\rm ht}(\delta)<{\rm ht}(\beta)$. 
 
 \tbf{Case 1}. First consider the case $\beta-\gamma_i\in \Phi^+$ for $i=1$ or $2$. 
 Fix such $i$ and write $\beta-\gamma_i = \gamma+\gamma'-\gamma_i$. 
 If $\gamma=\gamma_i$ or $\gamma'=\gamma_i$, then \eqref{eq:address} holds true trivially. 
We may assume $\gamma\ne\gamma_i$ and $\gamma'\ne\gamma_i$. 
 By Lemma \ref{lem:3roots}, either $\gamma-\gamma_i \in \Phi$ or $\gamma'-\gamma_i \in \Phi$. 
 Without loss of generality, assume $\gamma-\gamma_i \in \Phi$. 
 If $\gamma-\gamma_i$ is a positive root, then by applying the induction hypothesis to $\beta-\gamma_i= (\gamma-\gamma_i)+\gamma'$ we obtain 
 $$ r(\gamma-\gamma_i)+r(\gamma')-1 \le r(\beta-\gamma_i) \le r(\gamma-\gamma_i)+r(\gamma').$$
 This is equivalent to  \eqref{eq:address} and we are done. 
  If $\gamma-\gamma_i$  is a negative root, then $r(\gamma_i-\gamma)=r(\gamma)=1$ since $\gamma, \gamma_i-\gamma \le \gamma_i$. 
Now apply the induction hypothesis to $\gamma' = (\beta-\gamma_i)+(\gamma_i-\gamma) $ to obtain \eqref{eq:address}. 

Before going to the next case, let us address an observation.

\begin{observation} 
\label{ob:sigma} 
Let $\sigma \in \Phi^+$ be a positive root \tbf{covered by} $\beta$, i.e.~$\beta - \sigma \in \Delta$. 
Then for any choice of $ \gamma,\gamma'\in \Phi^+$ such that $\beta = \gamma+\gamma'$, we have the following estimation for $r(\sigma)$:
$$ r(\gamma)+r(\gamma')-2 \le r(\sigma) \le r(\gamma)+r(\gamma').
$$
In particular, the  estimation above holds true for $r(\beta)$ if $  r(\beta) = r(\sigma) $ for some root $\sigma$ covered by $\beta$.
\end{observation}

\begin{proof} [Proof of Observation \ref{ob:sigma}]
Write $ \alpha :=\beta - \sigma \in \Delta$.
  If $\gamma= \alpha $ or $\gamma'= \alpha $, then the observation holds true trivially. 
We may assume $\gamma\ne \alpha $ and $\gamma'\ne \alpha $. 
Apply  Lemma \ref{lem:3roots} to $\sigma = \gamma+\gamma'-\alpha  \in \Phi^+$ to obtain either $\gamma-\alpha \in \Phi^+$ or $\gamma'-\alpha \in \Phi^+$. 
 Without loss of generality, assume $\gamma-\alpha \in \Phi^+$. 
We apply the induction hypothesis to $\sigma$ and $\gamma$ for the expressions $\sigma = (\gamma-\alpha)+\gamma'$ and $\gamma= (\gamma-\alpha)+\alpha$, and we obtain the desired inequalities. 
\end{proof}
 
 \tbf{Case 2}.  It remains to consider $\beta-\gamma_i\notin \Phi^+$ for $i=1$ and $2$. 
By condition  \ref{lem:address}\eqref{lem:address3}, $r(\sigma) \le r(\beta)$ for any root $\sigma$ covered by $\beta$, and we may choose $\alpha_1 \in \Delta$ such that $\sigma_1:=\beta - \alpha_1  \in \Phi^+$ with $r(\beta)=r(\sigma_1)$. 
By Observation \ref{ob:sigma},
$$ r(\gamma)+r(\gamma')-2 \le r(\beta)=r(\sigma_1) \le r(\gamma)+r(\gamma').
$$
In particular, the upper bound of $r(\beta)$ in  \eqref{eq:address} follows. 

Suppose to the contrary that the lower bound does not hold, i.e.~$ r(\beta) =  r(\gamma)+r(\gamma')-2$. 
Again by Observation \ref{ob:sigma}, if $\sigma$ is any root covered by $\beta$, then
$$  r(\beta) = r(\sigma) =  r(\gamma)+r(\gamma')-2.
$$
This implies that  $\gamma\ne \alpha_1 $ and $\gamma'\ne \alpha_1 $. 
Apply  Lemma \ref{lem:3roots} to $\sigma_1 = \gamma+\gamma'-\alpha_1  \in \Phi^+$ to obtain  either $\gamma-\alpha_1 \in \Phi^+$ or $\gamma'-\alpha_1 \in \Phi^+$. 
 Without loss of generality, assume $\gamma-\alpha_1 \in \Phi^+$. 
 Now apply the induction hypothesis to $\sigma_1 = (\gamma-\alpha_1)+\gamma'$ and $\gamma= (\gamma-\alpha_1)+\alpha_1$ to obtain 
 $$r(\gamma)=r (\gamma-\alpha_1)+1.$$
 
 Let us recollect our assumptions and show the following claim. 
 We will find a contradiction after applying the claim repeatedly. 

\begin{claim} 
\label{cl:shift} 
Recall from the above that $  r(\beta) = r(\sigma) $ for  any root $\sigma$ covered by $\beta$. 
An ordered pair $\{\gamma, \gamma' \}$ of positive roots is called \tbf{bad} for $\beta$ if $\beta= \gamma+\gamma'$, $ r(\beta) =  r(\gamma)+r(\gamma')-2$, and there is $\alpha_1 \in \Delta$ such that $\gamma - \alpha_1  \in \Phi^+$ with $r(\gamma)=r (\gamma-\alpha_1)+1$. 
If  there exists a bad ordered pair $\{\gamma, \gamma' \}$  for $\beta$, then there exists another bad ordered pair $\{\mu, \mu' \}$ for $\beta$ with $\mu<\gamma$. 
\end{claim}

\begin{proof} [Proof of Claim \ref{cl:shift}]
 We consider two cases.
 
  \tbf{Subcase 1}. $\gamma-\gamma_i\notin \Phi^+$ for $i=1$ and $2$. 
  By condition  \ref{lem:address}\eqref{lem:address3}, we may choose $\alpha_2 \in \Delta$ such that $\mu :=\gamma - \alpha_2  \in \Phi^+$ and $r(\gamma)=r(\mu)$. 
  In particular, $ \alpha_1 \ne  \alpha_2$ since $r(\gamma - \alpha_2)\ne r (\gamma-\alpha_1)$. 
  It cannot happen that $\gamma =  \alpha_1+  \alpha_2$; otherwise, $r(\gamma)$ has both values $1$ and $2$, which is absurd. 
  By Lemma \ref{lem:proximity}, $\mu-\alpha_1 =\gamma -  \alpha_1-  \alpha_2  \in \Phi^+$.
By the induction hypothesis, 
  $$r(\gamma)-1 =r (\gamma-\alpha_1) \ge r (\mu-\alpha_1) \ge r (\mu)-1.$$
This follows that
$$r (\mu-\alpha_1) = r (\mu)-1.$$
  See Figure \ref{fig:sc1} on the left for an illustration of these roots in the root poset.
  
  Consider the expression $\beta = \mu + \gamma' + \alpha_2 \in \Phi^+$. 
  If $\beta - \alpha_2 =\mu + \gamma'  \in \Phi^+$, then by applying the induction hypothesis to $\beta - \alpha_2 $ we have that 
  $$ r(\beta) = r(\beta - \alpha_2) \ge r(\mu) + r(\gamma' )-1 = r(\gamma)+r(\gamma')-1.$$
This is a contradiction. 
Hence $\beta - \alpha_2 \notin \Phi^+$, and Lemma \ref{lem:3roots} forces $\mu':= \gamma' + \alpha_2 \in \Phi^+$. 
Use the lower bound of $r(\beta)$ in Observation \ref{ob:sigma} for the expression $\beta = \mu+\mu'$ to obtain
$$  r(\gamma)+r(\gamma')-2 =  r(\beta)  \ge  r(\mu)+r(\mu')-2.$$
Thus $r(\gamma') \ge r(\mu')$. 
However, $r(\gamma') \le r(\mu')$ by the induction hypothesis. 
Therefore, $r(\gamma')= r(\mu')$. 
This follows that $\{\mu, \mu' \}$ is a bad ordered pair for $\beta$ with $\mu<\gamma$ we wanted to find. 
\vskip .5em

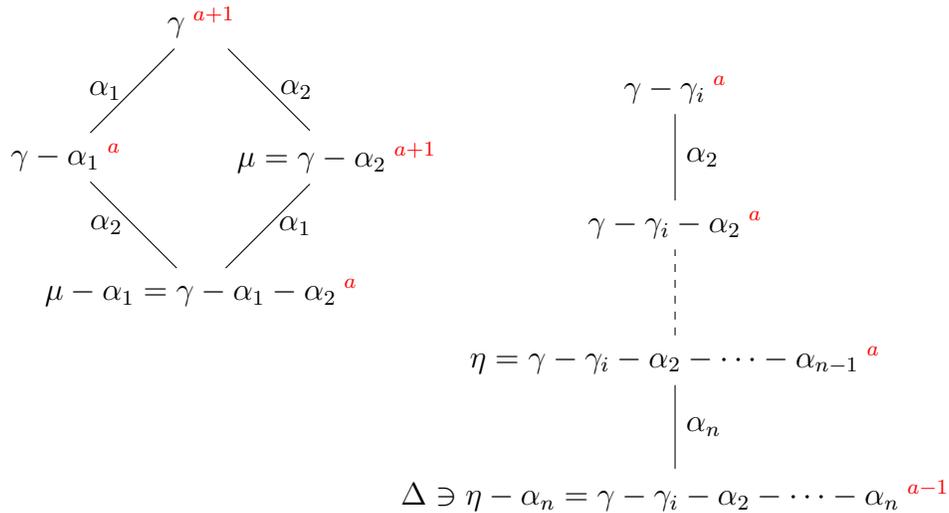
\begin{figure}[htbp!]
   \centering
 \begin{tikzpicture}[scale=.9]
 
  \node (1) at (-2,0) {$\gamma-\alpha_1$ \textcolor{red}{$^{a}$}};
  \node (2) at (2,0) {$\mu=\gamma-\alpha_2 $ \textcolor{red}{$^{a+1}$}};
  \node (0) at (0,2) {$\gamma $ \textcolor{red}{$^{a+1}$}};
   \node (12) at  (0,-2) {$\mu-\alpha_1 = \gamma-\alpha_1-\alpha_2  $ \textcolor{red}{$^{a}$}};
   
    \node (i) at (7,1) {$\gamma-\gamma_i$ \textcolor{red}{$^{a}$}};
        \node (i2) at (7,-1) {$\gamma-\gamma_i-\alpha_2$ \textcolor{red}{$^{a}$}};
        
                        \node (i2n-1) at (7,-3) {$\eta=\gamma-\gamma_i-\alpha_2-\cdots-\alpha_{n-1}$ \textcolor{red}{$^{a}$}};
                        
                \node (i2n) at (7,-5) {$\Delta \ni \eta- \alpha_n = \gamma-\gamma_i-\alpha_2-\cdots-\alpha_n$ \textcolor{red}{$^{a-1}$}};
 
  \draw (1)--(0) node [midway, left] {$\alpha_1$};
    \draw (2)--(12)  node [midway, right] {$\alpha_1$};
    \draw  (0)--(2) node [midway, right] {$\alpha_2$};
    \draw (1)--(12)  node [midway, left] {$\alpha_2$};
    
      \draw (i)--(i2) node [midway, right] {$\alpha_2$};
            \draw[dashed] (i2n-1)--(i2)  ;
                  \draw (i2n-1)--(i2n) node [midway, right] {$\alpha_n$};

\end{tikzpicture}
 \caption{A pictorial illustration of the proof of Claim \ref{cl:shift}. Each root $\alpha$ in the root poset is written next to the evaluation $r(\alpha)$ of the map $r$. The illustration for \tbf{Subcase 1} is on the left, the illustration for \tbf{Subcase 2} is the entire figure.}
\label{fig:sc1}
\end{figure}

  \tbf{Subcase 2}. $\gamma-\gamma_i\in \Phi^+$ for $i=1$ or $2$. 
   Fix such $i$ and write $\beta= (\gamma-\gamma_i)+\gamma'+\gamma_i$. 
Since  $\beta-\gamma_i\notin \Phi^+$, by Lemma \ref{lem:3roots},  $\gamma'+\gamma_i \in \Phi^+$. 
Suppose $\gamma-\gamma_i\in \Delta$. 
Then $r(\gamma)=2$ and $r(\gamma') = r(\beta)$. 
 However, $ r(\beta) = r(\gamma'+\gamma_i) =r(\gamma') +1$ since $\gamma'+\gamma_i$ is covered by $\beta$. 
 This is a contradiction. 
 Thus we may assume $\gamma-\gamma_i\in  \Phi^+ \setminus \Delta$. 
In particular,  ${\rm ht}(\gamma) >2$.
 
 If there exists $\alpha \in \Delta$ such that $\gamma-\gamma_i -\alpha\in \Phi^+$ and $r(\gamma-\gamma_i -\alpha) = r(\gamma)-2$, then we may choose $\{  \gamma-\gamma_i, \gamma'+\gamma_i\}$ as a desired bad ordered pair for $\beta$. 
 If not, by the induction hypothesis for any $\alpha_2 \in \Delta$ such that $\gamma-\gamma_i -\alpha_2\in \Phi^+$, we must have $r(\gamma-\gamma_i -\alpha_2) = r(\gamma)-1$. 
 Fix such an $\alpha_2 \in \Delta$.
If $\beta -\alpha_2\in \Phi^+$, then by applying the induction hypothesis to $\beta -\alpha_2=  (\gamma-\gamma_i -\alpha_2)+(\gamma'+\gamma_i)$ we obtain a contradiction. 
Hence $\beta - \alpha_2 \notin \Phi^+$, and Lemma \ref{lem:3roots} forces $  \gamma' + \gamma_i +\alpha_2 \in \Phi^+$. 
Use the lower bound of $r(\beta)$ in Observation \ref{ob:sigma} for the expression  $\beta=  (\gamma-\gamma_i -\alpha_2)+(\gamma'+\gamma_i +\alpha_2)$  to obtain
$$r( \gamma' + \gamma_i +\alpha_2) =r(\gamma') +1.$$
By a similar reason that $\gamma-\gamma_i\in  \Phi^+ \setminus \Delta$ as above, we also have $\gamma-\gamma_i -\alpha_2\in  \Phi^+ \setminus \Delta$. 

Furthermore, $\alpha_1 \ne \alpha_2$; otherwise,  $r(\gamma-\gamma_i -\alpha_1) = r(\gamma -\alpha_1)-1 =r(\gamma)-2$, which is absurd. 
 Apply  Lemma \ref{lem:3roots} to $\gamma-\gamma_i -\alpha_2\in \Phi^+$ to obtain either $\gamma-\alpha_2 \in \Phi^+$ or $\gamma_i +\alpha_2 \in \Phi^+$. 
 If $\gamma-\alpha_2 \in \Phi^+$, then 
 $$r(\gamma-\alpha_2) = r(\gamma-\gamma_i -\alpha_2)+1=r(\gamma).$$
   By Lemma \ref{lem:proximity}, $\gamma -  \alpha_1-  \alpha_2  \in \Phi^+$ since ${\rm ht}(\gamma) >2$. 
  This leads us to the ``diamond" poset on $4$ elements on the left of Figure \ref{fig:sc1}. 
  By a similar argument as in \tbf{Subcase 1}, we obtain $  \gamma' + \alpha_2 \in \Phi^+$ and we may choose $\{  \gamma-\alpha_2, \gamma'+\alpha_2\}$ as a desired bad ordered pair for $\beta$. 
   
   Suppose $\gamma_i +\alpha_2 \in \Phi^+$ and consider the expression $\beta  =  (\gamma-(\gamma_i +\alpha_2))+(\gamma'+(\gamma_i +\alpha_2 ))$. 
   By repeating the arguments in the preceding three paragraphs with $\gamma_i +\alpha_2$ in place of $\gamma_i$, we can find either a bad ordered pair fulfilling the requirement in  Claim \ref{cl:shift},  or a simple root $\alpha_3 \in \Delta \setminus \{\alpha_1\}$ such that $\beta - \alpha_3 \notin \Phi^+$, $\gamma-\gamma_i -\alpha_2-\alpha_3\in \Phi^+$ with $r(\gamma-\gamma_i -\alpha_2-\alpha_3) = r(\gamma)-1$, $\gamma'+\gamma_i +\alpha_2+\alpha_3\in \Phi^+$ with $r(\gamma'+\gamma_i +\alpha_2+\alpha_3) = r(\gamma')+1$, and $\gamma_i +\alpha_2+\alpha_3\in \Phi^+$. 
   The proof for this fact runs along the lines of the preceding paragraphs. 
   The only places we need further technique are to verify $\alpha_1 \ne \alpha_3$, and  if $\gamma-\alpha_3 \in \Phi^+$ then $r(\gamma-\alpha_3) =r(\gamma)$. 
 If $\alpha_1 = \alpha_3$, then by the induction hypothesis 
   $$r(\gamma-\gamma_i -\alpha_2-\alpha_3) = r(\gamma-\gamma_i -\alpha_2-\alpha_1) \le r(\gamma-\alpha_1)-r(\gamma_i +\alpha_2) +1  =  r(\gamma)-2,$$
   which is absurd.
  If $\gamma-\alpha_3 \in \Phi^+$, then again by the induction hypothesis 
 $$r(\gamma) \ge r(\gamma-\alpha_3) \ge r(\gamma-\gamma_i -\alpha_2-\alpha_3) + r(\gamma_i +\alpha_2) -1=r(\gamma).$$ 
 Hence $r(\gamma-\alpha_3) =r(\gamma)$. 
  
 Repeat for $\gamma_i +\alpha_2+\alpha_3$ in place of  $\gamma_i +\alpha_2 $ and so on. 
 This process of finding bad ordered pairs will have to terminate until we find simple roots $\alpha_2,\ldots,\alpha_{n-1} \in \Delta \setminus \{\alpha_1\}$ for $n \ge2$ such that $\beta - \alpha_j \notin \Phi^+$ for $2 \le j \le n-1$, 
 $\eta:=\gamma-\gamma_i-\alpha_2-\cdots-\alpha_{n-1}\in \Phi^+$ with ${\rm ht}(\eta)=2$ and $r(\eta) = r(\gamma)-1$, 
 $\eta':=\gamma'+\gamma_i +\alpha_2+\cdots+\alpha_{n-1}\in \Phi^+$ with $r(\eta') = r(\gamma')+1$.
 Moreover, there exists $\alpha_n \in \Delta$ such that $\eta- \alpha_n\in \Delta$ and $r(\eta- \alpha_n) = r(\gamma)-2$. 
 At each step of the process we always find a desired bad ordered pair for $\beta$, and $\{\eta,\eta'\}$ is the one in the final step. 
 This concludes the proof of Claim \ref{cl:shift}.
  
\end{proof}

 Now we return to the proof of \tbf{Case 2}. 
By the discussion before Claim \ref{cl:shift}, $\{\gamma, \gamma' \}$ is a bad ordered pair for $\beta$. 
Hence by Claim \ref{cl:shift}, there exist infinitely many mutually distinct bad ordered pairs for $\beta$. 
This is a contradiction since $\Phi^+$ is finite. 
This completes the proof of Lemma \ref{lem:address}.

\end{proof}

\begin{corollary}
\label{cor:addr}   
 Suppose we are in the situation of  Lemma \ref{lem:address}.
 For  $i=1$ or $2$, define a map $r_i:\Phi^+ \longrightarrow \Z$ by $r_i(\gamma_i)=2$ and $r_i(\beta) = r(\beta)$ for all $\beta \ne \gamma_i$. 
 Then   $r_i:\Phi^+ \longrightarrow \Z$ is the address of an alcove. 
 
\end{corollary} 
\begin{proof} 
We use Lemma \ref {lem:Shi}. 
We need to show for any $\beta \in \Phi^+$ and any choice of $\gamma,\gamma'\in \Phi^+$ such that $\beta = \gamma+\gamma'$, the following holds
$$
 r_i(\gamma)+r_i(\gamma')-1 \le r_i(\beta) \le r_i(\gamma)+r_i(\gamma').
$$
It suffices to show the above when $\gamma_i$ is involved, i.e.~$\gamma_i=\beta$ or  $\gamma_i= \gamma$ or $\gamma_i= \gamma'$. 
The rest is straightforward from the definition of the map $r$.
\end{proof}

\begin{lemma} 
\label{lem:(1)=>(2)} 
If a subset $\Sigma\subseteq\Phi^+$ is compatible, then it is $2$-locally compatible.
\end{lemma}

\begin{proof} 
 For non-triviality we may assume that the rank of $\Phi$ is at least $3$. 
Then any rank $2$ irreducible root subsystem   of $\Phi$ is of type $A_2$ or $B_2$. 
Let $X\in L_2(\A_{\Phi^+})$ and suppose that $\Phi_X$ is irreducible. 
Denote  $\Delta_X=\{\gamma_1,\gamma_2 \}$ for distinct $\gamma_1,\gamma_2 \in \Phi^+$. 
We need to show that  the localization $\Sigma_X=\Sigma \cap \Phi^+_X$ is compatible in $\Phi_X$. 
Since $\Phi_X=A_2$ or $B_2$, by Remark \ref{rem:LC-LS}(\ref{item:A2B2}), this is equivalent to showing that $\Sigma_X= \emptyset$ or $\Sigma_X \cap \Delta_X \ne \emptyset$.
Suppose not, then there exists $\beta \in \Sigma_X $ such that $\beta = d_1\gamma_1+d_2\gamma_2$ for $d_1,d_2 \in \Z_{> 0}$. 

Let $r=r_X:\Phi^+ \longrightarrow \Z$ be the map defined by $X$ from Lemma \ref{lem:address}, and $r_i: \Phi^+ \longrightarrow \Z$ for $i=1, 2$ be  the map   from Corollary \ref{cor:addr}. 
Fix $i$.
Let $A, A_i$ be the alcoves with addresses $r,r_i$, respectively. 
It is easily seen that $A \subseteq P^\diamondsuit$. 
Comparing the addresses of $A$ and $A_i$ implies that $H_{\gamma_i}^{1}$ is a wall of each alcove and separates these two. 
Hence $H_{\gamma_i}^{1}$ is a ceiling of $A$ for each $i=1, 2$.
We claim that 
$$A^\diamondsuit  \cap H_{\gamma_1}^{1}\cap H_{\gamma_2}^{1}  = A^\diamondsuit  \cap H_\beta^{d_1+d_2} .$$

Since any two facets of a simplex are adjacent, the intersection $R:=A^\diamondsuit  \cap H_{\gamma_1}^{1}\cap H_{\gamma_2}^{1}$ is a nonempty face of $A^\diamondsuit$. 
Moreover, it is contained in the face $Q:=A^\diamondsuit  \cap H_\beta^{d_1+d_2}$ of $A^\diamondsuit$ by the definition of $\beta$. 
In particular, $Q$ is not empty. 
Since any proper face of a polytope is the intersection of all facets containing it, if $Q=\bigcap_{\delta \in D} H_{\delta}^{n_\delta}\cap  A^\diamondsuit$ where each $ H_{\delta}^{n_\delta}$ is a ceiling of $A^\diamondsuit$ and $D$ is a set of positive roots, then $D\subseteq\{\gamma_1,\gamma_2 \}$ and $n_\delta=1$ for all $\delta \in D$. 
Applying Lemma \ref{lem:ideal-crucial} to the face $Q$ implies that $D=\{\gamma_1,\gamma_2 \}$ and $R=Q$. 

Since $\beta \in \Sigma$, by the compatibility of $\Sigma$, either $\gamma_1 \in \Sigma$ or  $\gamma_2 \in \Sigma$. 
Hence $\Sigma_X \cap \Delta_X \ne \emptyset$, a contradiction. 
This completes the proof. 
\end{proof}

\section{Proof of the second main result: Theorem \ref{thm:SF-characterize}}
\label{sec:proof2}

 First we recall some freeness properties of Weyl arrangements. 
 
 \begin{lemma} {\cite[Theorem 2]{AY09}}
\label{lem:shift-iso} 
 Let $k\in \Z_{>0}$, 
  Let $\A$ be the Weyl arrangement of an irreducible root system $\Phi$, and $m : \A \longrightarrow \{0, 1\}$ be a multiplicity. Then there exists an isomorphism of $S$-modules 
  $$D(\A,m) \longrightarrow  D(\A,2k + m).$$
  Here $2k + m$ means the  multiplicity that sends any $H\in\A$ to $2k+m(H)$.
\end{lemma}

\begin{lemma} {\cite[Claim in the proof of Theorem 1.6]{AT16}}
\label{lem:2-dim} 
Let $\Phi$ be an irreducible root system of rank $2$, and $\Sigma\subseteq\Phi^+$. 
Then $\Sigma$ is Shi-free, i.e.~the cone $\cc\scS^k_{\Sigma}$ is free for every $k>0$ if and only if $\Sigma=\emptyset$ or $\Sigma \cap \Delta \ne \emptyset$, equivalently, $\Sigma$ is $2$-locally simple.
 
\end{lemma}
 
 We are ready to give the proof of our second main result. 

\begin{proof}[\textbf{Proof of Theorem \ref{thm:SF-characterize}}] 
First we show  $(1) \Leftrightarrow (2)$. 
We need some notations. 
Let $H_{\infty}: z=0$  denote the hyperplane at infinity.
Let $\A:=\A_{\Phi^+}$ be the Weyl arrangement of  $\Phi$. 
Define a multiplicity $m : \A \longrightarrow \{0, 1\}$ by $m(H_\alpha)=1$ if $\alpha \in \Sigma$ and $m(H_\alpha)=0$ otherwise. 
Then the Weyl subarrangement $\A_\Sigma$ can be identified with the multiarrangement $(\A,m)$. 
Moreover, the Ziegler restriction of $\cc\scS^k_{\Sigma}$ onto $H_{\infty}$  can be identified with $ D(\A,2k + m)$. 
By Lemma \ref{lem:shift-iso}, $\Sigma$ is free, i.e.~$\A_\Sigma$  is free $  \Leftrightarrow D(\A,2k + m)$ is free. 

By Theorem \ref{thm:Yoshinaga's criterion}, it suffices to show that  $\Sigma$ is $2$-locally simple $  \Leftrightarrow \cc\scS^k_{\Sigma}$ is $3$-locally free along $H_{\infty}$, i.e.~
for any $X \in L_3(\cc\scS^k_{\Sigma})$ with $X \subseteq H_{\infty}$, the localization $(\cc\scS^k_{\Sigma})_X$ is free.
By  \cite[Lemma 3.1]{AT11} for any such $X$, there exists $Y \in L_2(\A)$ such that $X = Y \cap  H_{\infty}$. 
Moreover, the relation $X = Y \cap  H_{\infty}$ implies 
$$(\cc\scS^k_{\Sigma})_X =  \cc\scS^k_{\Sigma_Y} (\Phi_Y) \times \varnothing_Y,$$
where $\varnothing_Y$ denotes the empty arrangement in $Y$. 
Thus $(\cc\scS^k_{\Sigma})_X$ is free $  \Leftrightarrow \cc\scS^k_{\Sigma_Y} (\Phi_Y)$ is free. 

When $\Phi_Y$ is reducible, it is easily seen that $\cc\scS^k_{\Sigma_Y} (\Phi_Y)$ is always free. 
Otherwise, by Lemma \ref{lem:2-dim}, it is free if and only if   $\Sigma_Y=\emptyset$ or $\Sigma_Y \cap \Delta_Y \ne \emptyset$. 
Hence if  $\Sigma$ is $2$-locally simple, then $\cc\scS^k_{\Sigma}$ is $3$-locally free along $H_{\infty}$. 
Conversely, suppose that $\cc\scS^k_{\Sigma}$ is $3$-locally free along $H_{\infty}$ but $\Sigma$ is not $2$-locally simple. 
Then there exists $Y \in L_2(\A)$ such that $\Phi_Y$ is irreducible, $\Sigma_Y\ne\emptyset$ and $\Sigma_Y \cap \Delta_Y = \emptyset$. 
By Lemma \ref{lem:2-dim}, $   \cc\scS^k_{\Sigma_Y} (\Phi_Y)$ is not free. 
Define $X := Y \cap  H_{\infty}$. 
Then by the preceding paragraph,  $(\cc\scS^k_{\Sigma})_X$ is not free. 
This contradicts the $3$-local freeness along $H_{\infty}$ of $\cc\scS^k_{\Sigma}$. 

Next we show  $(2) \Leftrightarrow (3)$. 
Note that by Theorem \ref{thm:WC}, $\CO=\LC$ for any irreducible root system $\Phi$. 
Suppose $\Phi \ne G_2$. 
Then any rank $2$ irreducible root subsystem $\Phi'$ of $\Phi$ is of type $A_2$ or $B_2$. 
 By Remark \ref{rem:LC-LS}(\ref{item:A2B2}), $\LC=\LS$ in $\Phi'$. 
 Therefore, for any  $\Sigma\subseteq\Phi^+$, 
 $$\Sigma \in \CO \Leftrightarrow  \Sigma\in \LC \Leftrightarrow \Sigma\in \LS.$$

For $\Phi = G_2$, the equivalence $(2) \Leftrightarrow (3)$  is proved by a direct check in Figure \ref{fig:G2}. 
The main reason why the exceptional cases exist is that the affine hyperplanes orthogonal to the highest root $3\alpha_1+2\alpha_2$ have non-facet intersections with three different upper closed alcoves (the alcoves in green in Figure \ref{fig:G2}). 
Moreover, only one of these non-facet intersections can be lifted to a facet intersection by an affine hyperplane orthogonal to a short root. 
More precisely, the non-facet intersection at the point $\varpi^\vee_1 +\varpi^\vee_2$ of the alcove furthest away from the origin can be lifted to a facet supported by the ceiling $H_{\alpha_1}^1$. 
In particular, $\Sigma=\{\alpha_1, 3\alpha_1+2\alpha_2\}$ is not compatible. 
However, $\Sigma$  is $2$-locally simple since it contains the simple root $\alpha_1$.
\end{proof}

  We close this section by giving an example to illustrate the applicability of our Theorems \ref{thm:WC} and \ref{thm:SF-characterize}.  
  
  \begin{example}
\label{ex:F4}
Let $\Phi=F_4$ with $\Delta=\{\alpha_1,\alpha_2,\alpha_3,\alpha_4\}$. 
Suppose the Dynkin diagram of $\Phi$ is given by $\alpha_1\,\mbox{---}\,\alpha_2 \overset{\longrightarrow}{=\joinrel=} \alpha_3 \,\mbox{---}\, \alpha_4$ where $\alpha_1,\alpha_2$ are the long simple roots. 
 Let    $\Sigma_1=\{\alpha_2, \alpha_2+2\alpha_3\}$ and $\Sigma_2=\{\alpha_2, \alpha_2+2\alpha_3,\alpha_1+\alpha_2+\alpha_3+\alpha_4\}$. 
We will check the compatibility of $\Sigma_1$ and $\Sigma_2$ by using  Theorem \ref{thm:WC}. 
In the type $F_4$ case, $\CO = \NC$ so it suffices to check the negative coclosedness of these sets. 
The set  $\Sigma_1$ is negatively coclosed by the same reason as in the type $B_2$ example in Remark \ref{rem:LC-LS}(\ref{item:NC-CC}). 
However, $\Sigma_2$ is not negatively coclosed since  $\alpha_1+\alpha_2+\alpha_3+\alpha_4=    (\alpha_1+\alpha_2)+(\alpha_3+\alpha_4)$ and $ (\alpha_1+\alpha_2, \alpha_3+\alpha_4)=(\alpha_2, \alpha_3)<0$ but neither $\alpha_1+\alpha_2$ nor $\alpha_3+\alpha_4$ belongs to $\Sigma_2$. 
As a result, $\Sigma_1$ is Shi-free while $\Sigma_2$ is not by Theorem \ref{thm:SF-characterize}. 

The compatibility of $\Sigma_1$ and $\Sigma_2$ is, however, very complicated to check by using  Definition \ref{def:compatibleW} or Theorem \ref{thm:CO}. 
The former requires the information of how the affine hyperplanes orthogonal to the roots in $\Sigma_1$ or $\Sigma_2$ intersect with the upper closed alcoves in $P^\diamondsuit$. 
This is difficult to see in dimension $4$ (or higher). 
The latter requires the calculation on the characteristic, Ehrhart quasi-polynomials and $\A$-Eulerian polynomial. 
Computing the $\A$-Eulerian polynomial is enormous: We need in principle to compute the descent of every Weyl group element (in this case, $|W(F_4)|=1152$).
For this reason, even if we know a subset $\Sigma\subseteq\Phi^+$ is compatible, it is still difficult to compute the characteristic quasi-polynomial of $\Sigma$ by using the formulas in Theorem \ref{thm:CO}.
\end{example}

\section{Further comments and remarks}
\label{sec:remark}

In this section we address some further comments and remarks.
\begin{enumerate}[(A)]
\item 
In Theorem \ref{thm:SF-characterize}, we discussed the freeness of $\cc\scS^k_{\Sigma}$ and $\A_\Sigma$, and showed that they are closely related. 
For a given free arrangement, there is another important algebraic invariant called the \emph{exponents}. 
More precisely, when an arrangement $\B$ is free, we may choose a homogeneous basis $\{\theta_1, \ldots, \theta_\ell\}$ for $D(\B)$.
Then the degrees of the $\theta_i$'s are called the \tbf{exponents} of $\B$ \cite[Definition 4.25]{OT92}. 
The exponents of $\cc\scS^k_{\Sigma}$ and $\A_\Sigma$ are also  closely related via the following theorem of Abe-Terao \cite{AT16}.

\begin{theorem}{{\cite[Theorem 1.6]{AT16}}}
 \label{thm:exp}
 Let $\h$ denote the Coxeter number of an irreducible root system $\Phi$.
Let $k \in \Z_{>0}$ and $\Sigma\subseteq\Phi^+$. 
 Then, $\cc\scS^k_{\Sigma}$  is free with exponents $(1, k\h + e_1, \ldots, k\h + e_\ell)$ if and only if $\cc\scS^k_{-\Sigma}$ is free with exponents $(1, k\h -e_1, \ldots, k\h - e_\ell)$. 
 In this case, $\A_\Sigma$  is also free with exponents $(e_1,\ldots, e_\ell)$.
\end{theorem}

When $\Sigma$ is an ideal of $\Phi^+$, there exists a nice combinatorial description of the exponents of $\A_\Sigma$ in terms of the dual partition of the \emph{height distribution} of $\Sigma$  \cite[Theorem 1.1]{ABCHT16}. 
It would be interesting to find a combinatorial description of the exponents of $\A_\Sigma$ when $\Sigma$ is both free and compatible.

\vskip .5em 

\item Given a coclosed subset  $\Sigma\subseteq\Phi^+$, Slofstra \cite{Sl16} showed that verifying the freeness of $\Sigma$ amounts to verifying the freeness of all  localizations on the flats of codimension at most $4$.
\begin{theorem}{{\cite[Theorem 3.1]{Sl16}}}
 \label{thm:cr4}
 Let $\Sigma\subseteq\Phi^+$ be a coclosed subset. 
Then  $\A_\Sigma$ is free if and only if the localization  $(\A_\Sigma)_X$
 is free for every $X \in L_p(\A_\Sigma)$ of codimension $p \le 4$.
\end{theorem}
We conjecture that the theorem above can be extended to compatible subsets. 
 
\begin{conjecture} 
 \label{conj:comp}
 Let $\Sigma\subseteq\Phi^+$ be a compatible subset. 
Then  $\A_\Sigma$ is free if and only if the localization  $(\A_\Sigma)_X$
 is free for every $X \in L_p(\A_\Sigma)$ of codimension $p \le 4$.
\end{conjecture}
Suppose $\Phi$ is a simply-laced root system.
By  Remark \ref{rem:LC-LS}(\ref{item:NC-CC}) and  Theorem \ref{thm:WC}, $\CC = \NC=\CO$. 
Then Conjecture \ref{conj:comp} is equivalent to Theorem \ref{thm:cr4}. 
For root systems of rank $\le 4$, there is nothing to be done.
Thus the conjecture remains open only when $\Phi$ is of type $B_\ell$ or $C_\ell$ for $\ell \ge 5$.

\vskip .5em 

\item In the case of type $A$, the  Shi-freeness of an arbitrary subset $\Sigma\subseteq\Phi^+$ can be characterized by combinatorial properties of graphs.
Let $\{\epsilon_1, \ldots, \epsilon_{\ell}\}$ be an orthonormal basis for  $V=\R^\ell$. 
Define $U : = \left\{ \sum_{i=1}^{\ell} x_i\epsilon_i  \in V \,\middle\vert\, \sum_{i=1}^{\ell} x_i=0\right\} \simeq \R^{\ell-1}$. 
 The set 
 $$\Phi(A_{\ell-1}) = \{\pm(\epsilon_i - \epsilon_j) \mid 1 \le i<j \le \ell\}$$
is a root system of type $A_{\ell-1}$ in $U$ with a positive system
$$
\Phi^+(A_{\ell-1}) =  \{ \epsilon_i-\epsilon_j \mid 1 \le i <  j \le \ell\}.
$$
Let $G$ be a simple graph (i.e.~no loops and no multiple edges) with vertex set $V_G = [\ell]:=\{1,2,\ldots,\ell\}$ and edge set $E_G$. 
Define 
$$\Sigma(G) := \{  \epsilon_i-\epsilon_j  \mid \{i,j\}  \in E_G \,( i <  j)\} \subseteq \Phi^+(A_{\ell-1}).$$ 
Thus any subset of $\Phi^+(A_{\ell-1})$ is completely defined by a simple graph $G$. 
The corresponding Weyl subarrangement $\A_{\Sigma(G)}$ is known as the \textbf{graphic arrangement}.
By  Theorem \ref{thm:SF-characterize} and \cite[Corollary 15]{TT21}, we have the following graphic characterization for the  Shi-freeness of $\Sigma(G)$.

\begin{corollary}
 \label{cor:interval}
 Let  $G = (V_G,E_G)$ be a graph with $|V_G|=\ell$.  The following  are equivalent. 
\begin{enumerate}[(i)] 
\item $G$ has a vertex-labeling using $[\ell]$ so that $\Sigma(G)$ is Shi-free. 
\item $G$ has a vertex-labeling using $[\ell]$ so that $\Sigma(G)$ is compatible and free. 
\item  $G$ is an interval graph. 
    \end{enumerate}
\end{corollary}

Interval graphs have several different characterizations. 
Among others, one of the relevant  characterizations is that the graph has an ordering  $v_1    < \cdots  < v_\ell$ of its vertices such that  if $ i < k < j$ and $\{v_i, v_j\}$ is an edge, then $\{v_i, v_k\}$ is an edge.
\end{enumerate}
 
\vskip 1em
\vn 
\textbf{Acknowledgements.} 
The first author is partially supported by JSPS KAKENHI grant numbers JP18KK0389 and 23K17298.

 \bibliographystyle{abbrv}
\bibliography{references}

\end{document}